\documentclass[11pt,epsfig]{article}

\usepackage{amsmath, amssymb, amsthm}
\usepackage[margin=1in]{geometry}

\newtheorem{thm}{Theorem}[section]
 \newtheorem{lem}[thm]{Lemma}
 \newtheorem{prop}[thm]{Proposition}
 
\theoremstyle{remark}
\newtheorem{rem}[thm]{Remark}
\newtheorem*{rem*}{Remark}
\theoremstyle{definition}

\title 
{Stability analysis for
the Implicit-Explicit discretization of the Cahn-Hilliard equation}

\author
{Dong Li
\thanks
{Department of Mathematics, the Hong Kong University of Science \& Technology,
Clear Water Bay, Kowloon, Hong Kong. Email: {mpdongli@gmail.com}.
The author's work was supported in part by Hong Kong RGC grant GRF 16307317
 and 16309518. }\qquad
{Tao Tang}
\thanks{SUSTech International Center for Mathematics, Shenzhen, China; and
    Division of Science and Technology, BNU-HKBU United International College,
    Zhuhai, Guangdong Province, China.
    Email: tangt@sustech.edu.cn. This author's work is partially supported by the NSF of China under grant number 11731006 and the science challenge project (No. TZ2018001).}
}

\begin{document}
\maketitle
\begin{abstract}
Implicit-Explicit methods have been widely used for the efficient numerical simulation of phase field
problems such as the Cahn-Hilliard equation or thin film type equations. Due to the lack of maximum
principle and stiffness caused by the effect of small dissipation coefficient, most existing theoretical analysis relies on adding additional stabilization terms, mollifying the nonlinearity or introducing auxiliary variables which implicitly either changes the structure of the problem or trades accuracy for stability
in a subtle way. In this work 
we introduce a  robust theoretical framework to analyze directly
the stability of the standard implicit-explicit approach 
without stabilization or any other modification. We take 
 the Cahn-Hilliard equation as a model case and prove energy stability under natural time step constraints which are optimal
with respect to energy scaling.  These settle several questions which have been open
since the  work of Chen and Shen \cite{CS98}.

\end{abstract}
\section{Introduction}
The Cahn-Hilliard (CH) equation was first introduced by Cahn and Hilliard in
\cite{CH58} to describe the complicated phase separation and coarsening phenomena
in non-uniform systems such as glasses, alloys and polymer mixtures.
In this work we are concerned with the numerical solutions for  the Cahn-Hilliard 
equation in nondimensionalized form as
\begin{align} \label{1}
\begin{cases}
\partial_t u  = \Delta ( -\nu \Delta u +f (u) ), \quad (t,x) \in \Omega \times (0,\infty) \\
u \Bigr|_{t=0} =u_0,
\end{cases}
\end{align}
where $u=u(t,x)$ is a real-valued function which represents the concentration difference
in a binary system, and $\nu>0$ is usually called mobility coefficient. The function
$f(u)$ is taken as the derivative of a standard double well potential:
\begin{align*}
f(u)=u^3- u = F^{\prime}(u), \quad F(u) = \frac 14 (u^2-1)^2.
\end{align*}
Note that with this choice the equation \eqref{1} is invariant under the sign change
$u \to -u$ which is natural since $u$ corresponds to the difference of concentrations.
The minima of the double well potential are $u=\pm 1$ which typically correspond to
the formation of domains. The length scale of the transition regions between domains is
typically proportional to $\sqrt{\nu}$. For simplicity we shall fix the spatial domain $\Omega$ in \eqref{1} as the usual $1$-periodic torus $\mathbb T^d=\mathbb R^d / \mathbb Z^d
=[-\frac 12,\frac 12)^d$ in physical dimensions $d\le 3$. With this it is convenient to work with
the  Fourier basis $(e^{2\pi i k\cdot x})_{k\in \mathbb Z^d}$ and this is the convention we shall adopt for the Fourier
inversion formulae to be used throughout this paper. For smooth solutions, the mass conservation
law takes the form
\begin{align*}
\frac d {dt} M(t) = \frac d {dt} \int_{\Omega} u(t,x) dx  \equiv 0.
\end{align*}
In particular $M(t) \equiv 0$ if the initial total mass is zero. If $M(t) \equiv c$ for some
$c$ nonzero, then by a change of variable $u\to u+c$ one can work with a modified
nonlinearity $f(u+c)$ in \eqref{1} and the corresponding analysis can be adjusted suitably.
Therefore throughout this work we will only consider mean zero initial data
for simplicity. 
As is well known the system \eqref{1} is a gradient flow of a Ginzburg-Landau type energy
functional $\mathcal E(u)$ in $H^{-1}$, i.e.
\begin{align*}
& \partial_t u= - \frac {\delta \mathcal E} {\delta u } \Bigr|_{H^{-1}} =
\Delta (\frac {\delta \mathcal E} {\delta u} ), \\
\end{align*}
where $\frac {\delta \mathcal E}{\delta u} \Bigr|_{H^{-1}}$ , $\frac {\delta \mathcal E}
{\delta u}$ denote the standard variational derivatives in $H^{-1}$ and $L^2$ respectively, and
\begin{align*}
&\mathcal E(u)= \int_{\Omega} ( \frac 12 \nu |\nabla u|^2 + F(u) ) dx
=\int_{\Omega} ( \frac 12 \nu |\nabla u|^2 + \frac 14 (u^2-1)^2 ) dx.
\end{align*}
As such the basic energy conservation law takes the form
\begin{align*}
\frac d {dt} \mathcal E ( u(t) ) +
\| |\nabla|^{-1}  \partial_t u \|_2^2
=\frac d {dt} \mathcal E(u(t)) + \int_{\Omega}
| \nabla ( -\nu \Delta u + f(u ) ) |^2 dx =0.
\end{align*}
Here one should note that $\partial_t u$ has mean zero so that $|\nabla|^{-1}$ is certainly 
well-defined.  From energy conservation one immediately deduces monotonic energy decay and
a priori $\dot H^1$ bound of the solution as
\begin{align*}
&\mathcal E(u(t) ) \le \mathcal E(u(s)), \qquad \forall\, t\ge s;\\
& \| \nabla u(t) \|_2 \le \sqrt{\frac 2 {\nu}}  \mathcal E(u(t)) 
 \le \sqrt{\frac 2 {\nu}} \mathcal E(u_0 ), \qquad \forall\, t>0.
\end{align*}
 By a scaling analysis one can identify  the critical spaces for CH in 2D and 3D as
  $L^2$ and $\dot H^{\frac 12}$ respectively. Global wellposedness in $H^1$ and regularity of solutions then follow easily
 from the a priori $\dot H^1$ bound and standard arguments. As the primary goal of the CH model
 is to understand the physics and dynamics of spinodal decomposition (especially concerning
 the stages after quenching), a large body of existing
 mathematical analysis is naturally 
 devoted to the investigation of asymptotic behavior of solutions concerning
 coarsening, pattern formation, evolution of interfaces, stability and instability in various
 time regimes.  One should keep in mind that the dynamics of CH is quite complex and 
 takes place on a myriad of time scales ranging from short time scales $t= O(\sqrt {\nu})$
 on which the solution generally develops rich structures with many internal layers and sharp
 gradients   to metastable time scales $t= O(e^{\operatorname{const}/\sqrt{\nu}})$
 on which the unstable or super slow modes are fully developed and 
  coarsening starts to dominate. As a result the numerical
 simulation of CH can be quite stiff as one need to resolve layers of order $\sqrt{\nu}$ especially
 when $\nu\ll 1$. 

There exist some natural scaling transformations which lead to slightly 
different forms  of the CH system in the literature. A sample case considered
in Liu and Shen \cite{LShen13} can be expressed as
\begin{align} \label{1a}
\partial_s \tilde u= \alpha \Delta ( -\Delta \tilde u + \beta f(\tilde u) ),
\quad (s, y) \in (0,\infty) \times \Omega_L, 
\end{align}
where $\alpha>0$, $\beta>0$, $\tilde u= \tilde u(s, y)$, and $\Omega_L=[-\frac L2, \frac L2)^d$. 
Now make a change of variable
\begin{align*}
\tilde u(s,y) = u(\frac s {\lambda}, \frac y L) = u(t,x), \qquad \lambda =\frac {L^2}{\alpha\beta},
\end{align*}
then we can recast \eqref{1a} into our standard form as
\begin{align} \label{1b}
\partial_t u = \Delta ( - \nu \Delta u + f(u) ), \quad (t,x) \in (0,\infty) \times \mathbb T^d, 
\end{align}
where $\nu = \frac 1 {\beta L^2}$. If we denote the typical time step for \eqref{1a} as
$\Delta s$ and for \eqref{1b} as $\Delta t$, then apparently we have 
\begin{align*}
&\Delta t = \frac {\Delta s }{\lambda} = \Delta s \cdot  {\alpha \beta} \cdot {L^{-2} }, \qquad
\nu = \frac {1}{\beta} \cdot L^{-2}.
\end{align*}
We record these relations here so that one can  translate many existing numerical analysis results in the literature about
the system \eqref{1a} in terms of the system \eqref{1b} which is quite convenient for comparison
purposes.

On the numerical side, there is  by now an enormous body of literature on the 
simulation and analysis of the CH equation and related phase field
models (cf. \cite{DN91, CS98,CJPWW14,Du09,GHu11,HLT07, 
SY10,ZCST99, Xu18} 
and the references therein). A fundamental challenge is to design fast, efficient and
accurate numerical schemes which are robust and energy stable especially in the computationally
stiff small $\nu$ regime. Roughly speaking almost all existing numerical algorithms can be classified into
five categories with some possible overlaps. To elucidate the discussion below, it is useful to
recast \eqref{1} as a abstract and more general model:
\begin{align*}
\partial_t u= \mathcal L u + \mathcal N(u),
\end{align*}
where $\mathcal L$ denotes the linear operator and $\mathcal N(u)$ collects the nonlinear part.
In its most generality one should think of $\mathcal N (u)$ as a functional. For example we
write $G(u, \nabla u, \nabla^2 u,\cdots)$ (for some function $G$) as $\mathcal N(u)$. 
To ease the discussion we shall ignore completely the space discretization and focus 
momentarily on first order
in time discretization. Now for $n=0,1,2,\cdots$, denote by $u^n\approx u(n\Delta t)$ as the numerical solution at  time  $t= nk$ where
$\Delta t$ is the time step size. Then the following are the prototypical schemes
often considered in the literature:
\begin{align*}
\frac{u^{n+1}- u^n}{k}=
\begin{cases}
\mathcal L u^n + \mathcal N(u^n), \qquad &\text{(Explicit)}, \\
\mathcal L u^{n+1} + \mathcal N (u^{n+1}), \qquad &\text{(Fully implicit)}, \\
\mathcal L u^{n+1} + \mathcal N (u^n), \qquad &\text{(Semi-implicit/Implicit-Explicit)}, \\
\mathcal L u^{n+1} + \mathcal N_{\operatorname{I}}
(u^{n}, u^{n+1}), \qquad &\text{(Partially implicit)}, \\
\mathcal L u^{n+1} + \widetilde{ \mathcal N}_{\operatorname{I}}(u^n, u^{n+1})
+\mathcal S(u^n,u^{n+1}), \qquad &\text{(Stabilization)}.
\end{cases}
\end{align*}
In the above $\mathcal N_{\operatorname{ I}}$ represents a careful splitting/interpolation of the 
nonlinearity term using $u^n$ and $u^{n+1}$.
To ensure consistency it should satisfy $\mathcal N_{ \operatorname{I}} (u,u)=
\mathcal N (u)$. In practical algorithms such as convex splitting one often convexify the 
problem and choose
 $\mathcal N_{\operatorname{I}}(u^n,u^{n+1}) = \mathcal N_+(u
)-\mathcal N_-(u)=\mathcal N(u)$ where $\mathcal N_+$ and $\mathcal N_-$
are convex. The term $\mathcal S(u^n, u^{n+1})$ represents certain
carefully chosen additional stabilization terms which vanishes suitably fast as the time
step $k\to 0$. In the last one we use the general notation $\widetilde{ \mathcal N}_{\operatorname{I}}(u^n, u^{n+1})$
to include $\mathcal N(u^n)$ or $\mathcal N_{\operatorname{I}}(u^n,u^{n+1})$ as 
special cases.

As mentioned above, the first type in our classification is pure explicit methods such as forward
Euler in time and explicit treatment of the linear dissipation and nonlinearity.   For small systems and short time scales, 
these methods are speedy, efficient and relatively easy to implement. But due to the poor stability and low accuracy one
often has to employ very small time step and spatial grid size which puts a serious limitation for large
scale and long time simulations. The second is (fully) implicit schemes such as forward Euler
in time and fully implicit of both the linear dissipation and the nonlinearity.  The Crank-Nicolson 
(CN) and
modified Crank-Nilcolson type methods also fall into this category. A representative work for the
CH equation in this direction dates back to Du and Nicolaides \cite{DN91} which analyzed both a semi-discrete fully implicit  in space with continuous time and a modified CN 
(see also the work of Elliott and Stuart \cite{ES93} pp. 1644  for the idea
of using secant approximation)
fully discrete scheme 
for the 1D CH system with Dirichlet boundary conditions. The deficiency of fully implicit
methods is the severe restriction on the time step in order to ensure solvability and the expense
of Newton's method for which efficient preconditioning is often needed in practice. Besides, for the CN
type schemes the nonlinearity often has to be modified suitably in order to ensure energy decay.

The third category is the semi-implicit methods which treats the principal linear dissipation term implicitly
and the nonlinear term explicitly.  In the phase-field context such methods date back to the
 work of Chen and Shen \cite{CS98} in which a semi-implicit Fourier spectral method
was implemented on a Allen-Cahn system and a CH system. These methods are quite efficient and 
accurate and  observed to have  good stability properties  in practical numerical simulations. 
However due to the lack of maximum principle and stiffness caused by small viscosity coefficient 
a rigorous stability and error analysis was a longstanding open problem. To get around this issue many
stabilized methods have been developed over the past decades which we will discuss in more detail
in the fifth category below.

The fourth group in our classification contains partially implicit methods. 
  These are one of the most explored directions during the past decades. The most popular ones
  are the convex-splitting schemes (CSS) which have been developed in \cite{Eyre98a, Eyre98b, 
  Guan14a, Guan14b, Guan16, FTY15, CJPWW14} for
  the CH model, higher order models and related nonlocal versions. The advantages of a typical
   CSS  are two: 1) Unconditional energy stability with no stringent restriction on the time step;
   2) Guaranteed convergence of the Newton iteration and relatively easy solvability of the associated
   nonlinear system.  This is in stark contrast with a standard fully implicit scheme where very small
   time steps need to be taken in order to ensure energy stability.
   
    However recently Xu et al in 
   \cite{Xu18} discovered a surprising reformulation of many CSS and stabilized 
   schemes as a version of the fully implicit scheme with a proper time re-parametrization/rescaling.
   As such it was argued that these methods implicitly trade numerical accuracy for 
   stability.  The fifth category in our classification consists of stabilized or mollified methods. 
The basic idea of stabilization is to introduce an additional $O(\Delta t^p)$ (for a $p^{th}$-order
method) term to the numerical scheme to alleviate the time step constraint. These methods
were first developed in \cite{ZCST99} for a Cahn-Hilliard-Cook equation, 
\cite{HLT07} for CH and \cite{XT06} for epitaxial growth models. The work of \cite{HLT07} and
\cite{XT06} relies on some conditional $L^{\infty}$-bound of the numerical solution.
In \cite{SY10} Shen and
Yang considered a modified/mollified CH system with suitable Lipschitz truncation and proved
various stability results under such assumptions. Removing these conditional assumptions and
proving the unconditional energy stability for such stabilized methods were known as
the unconditional stability conjecture. Recently in a series of papers
\cite{LQT16, LQT17, LQ17, LQ173d} several new methods were developed to settle the
unconditional stability conjecture for the 2D and 3D CH systems, including both first order and second
order in time methods.  Developing upon the second-order scheme in \cite{LQ17}, Song and
Shu \cite{SShu17} recently constructed a new unconditionally stable second order stabilized 
semi-implicit local discontinuous Galerkin method for the CH equation. 
In another direction Shen et al (see \cite{SXY19} and the references therein) fashioned 
another novel form of stabilization which is based on the introduction of an auxiliary variable.
A nice feature of this novel workaround is that it can render unconditional energy stability more easily.
However a new challenging issue is how to navigate properly the dynamics of
the fictitious variable in order to minimize its deviation from the true
dynamics.  On the practical side in order for
  the fictitious variable to stay close to a constant value, one has to
monitor very carefully the fluctuations of the auxiliary variable and even adaptively adjust time
steps in practical simulations. 
For a more detailed account of this and other more recent
algorithms and developments, we refer to \cite{SXY19, Tang18, Xu18} and the
references therein.  

The comprehensive stability and error analysis in \cite{LQT17, LQ17, LQ173d} shows that
the incorporation of additional stabilization terms in the numerical schemes 
does increase the stability of the algorithm, however it also introduces undesirable approximation
errors which may deteriorate accuracy in the long run. This phenomenon is also inherently present in the auxiliary variable approach \cite{SXY19} and accords well with the point
of view advocated by Xu et al in \cite{Xu18} which shows that there exists a subtle and 
fundamental balance
between stability and accuracy.  All these naturally lead us to wonder whether the sole pursuit of
unconditional energy stability whilst losing accuracy is worthwhile, and perhaps one should 
look for some sort of conditional stability with affordable time step constraints, and more importantly without sacrificing accuracy too much. In this perspective a fundamental unsettled 
issue since the  work of Chen and Shen \cite{CS98} is the identification of optimal time step constraints and  a rigorous stability  analysis of
 the original semi-implicit scheme  without any stabilization, mollification 
 or auxiliary variables. Indeed the very purpose of this work is to settle this important problem in the affirmative. 
We now state the main results. 

Consider the following semi-implicit Fourier-spectral discretization of \eqref{1}
on $\mathbb T^d=[-\frac 12, \frac 12)^d$ ($d\le 3$):
\begin{align} \label{sem1}
\begin{cases}
\displaystyle\frac{u^{n+1}-u^n}{\tau} = - \nu \Delta^2 u^{n+1} 
+\Delta \Pi_N ( f(u^n) ), \quad n\ge 0,
\\
u^0= \Pi_N u_0.
\end{cases}
\end{align}
where $\tau>0$ is the time step.
For each integer $N\ge 2$, define
\begin{align*}
X_N= \operatorname{Span}\Bigl\{ \cos(2\pi k\cdot x),\, \sin(2\pi k\cdot x)
:\quad k=(k_1,
\cdots, k_d)\in \mathbb Z^d, \; |k|_{\infty}=\max\{|k_1|, \cdots, |k_d| \} \le N
\Bigr\}.
\end{align*}
Note that the space $X_N$ includes the constant function. By a minor adjustment of the analysis
one can also consider the following space 
\begin{align*}
\widetilde{X_N}=\operatorname{Span}
\Bigl\{ \cos(2\pi k\cdot x),\, \sin(2\pi k\cdot x): \, k=(k_1,\cdots, k_d) \in \mathbb Z^d,\;
{\textstyle{-\frac N 2}} \le k_j \le \textstyle{\frac N2 -1} \text{ for all $1\le j\le d$} \Bigr\}
\end{align*}
which is more often used in practical computations especially when $N$ is a dyadic number so that
FFT can be implemented.
We define
the $L^2$ projection operator $\Pi_N: \, L^2(\Omega) \to X_N$  by
\begin{align}
(\Pi_N u-u, \phi)=0, \qquad \forall\,\phi \in X_N,
\end{align}
where $(\cdot,\cdot)$ denotes the usual $L^2$ inner product (for real-valued
functions) on $\Omega$. In yet other words, the operator $\Pi_N$ is simply
the truncation of Fourier modes  to  the frequency sector 
$\{|k|_{\infty}\le N\}$. Since  $u^0=\Pi_N u_0 \in X_N$,
by a simple induction one can check that  $u^n \in X_N$
for all $n\ge 0$. It is also possible to reformulate \eqref{sem1} in terms of
 the usual weak formulation, for example:
\begin{align*}
(\frac{u^{n+1}-u^n} {\tau}, v) 
+ ( \nabla (f(u^n)), \nabla v) + \nu (\Delta u^{n+1}, \Delta v)=0,
\quad \forall\, v \in X_N.
\end{align*}
However in our analysis it is slightly more convenient to work with \eqref{sem1}.
Note that  $u^n$ has mean zero  for all $n\ge 0$ since we assume $u_0$ has mean zero.

The following proposition albeit conditional is instrumental to
 understand the relationship between the time step and the $L^{\infty}$-norm of the numerical
 solution. 

\begin{prop}[Conditional stability for semi-implicit discretization, practical version] \label{propN1}
Let $d\le 3$, $\nu>0$, $\tau>0$ and $N\ge 2$. Assume $u_0 \in H^1(\mathbb T^d)$ and has mean zero.  Suppose up to $n=N_1$ the time step $\tau>0$ satisfies
\begin{align*}
\sqrt{\frac {2\nu}{\tau}} \ge 
\frac 32 \max_{0\le n \le N_1} \|u^n\|_{L^{\infty}(\mathbb T^d)}^2 - \frac 12.
\end{align*}
Then the semi-implicit scheme \eqref{sem1} is conditionally energy stable up to $n=N_1$, i.e.
\begin{align*}
\mathcal E(u^{n+1}) \le \mathcal E(u^n), \qquad\forall\,  0\le n \le N_1.
\end{align*}

\end{prop}

\begin{proof}[Proof of Proposition \ref{propN1}]
This follows directly from the discrete energy estimate Lemma \ref{z2}.
\end{proof}

In practical numerical simulations, it is observed that typical numerical solutions satisfy
$\|u^n\|_{\infty} = O(1)$ even for small $\nu \ll 1$.  As such if we assume the boundedness of the
numerical solution then by Proposition \ref{propN1},  the time step constraint
for $\tau$ is roughly $\tau \lesssim \nu$. In this sense Proposition \ref{propN1} is already useful for
guiding practical numerical simulations. On the other hand, even for the PDE exact solution there is
no $\nu$-independent $L^{\infty}$ bound for the nonlinearity without any mollification. 
Therefore some trade-off must be made in order to obtain a  stability result without assuming
the boundedness of the numerical solution. Our next theorem is a first result in this direction.

\begin{thm}[Conditional stability for semi-implicit discretization] \label{thm1}
Let $d\le 3$, $\nu>0$, $\tau>0$ and $N\ge 2$. Assume $u_0 \in H^1(\mathbb T^d)$ and has mean zero.  Assume $  \|u^0\|_{L^{\infty}(\mathbb T^d)} = L_0<\infty$ (recall $u^0=\Pi_N u_0$
and we may assume $L_0 \ne 0$). 
Then the semi-implicit scheme \eqref{sem1} is conditionally energy stable, i.e.
\begin{align*}
\mathcal E(u^{n+1}) \le \mathcal E(u^n), \qquad\forall\, n\ge 0,
\end{align*}
provided the following time step constraint is satisfied:
\begin{align*}
0<\tau\le \tau_{\operatorname{max}}
= \min\Bigl\{ \frac {8 \nu} {9 L_0^4}, \, 
\tau_{\operatorname{max}}^{(1)} \Bigr\},
\end{align*}
where
\begin{align*}
\tau_{\operatorname{max}}^{(1)}
=
\begin{cases}
C_1  \nu^{\frac 53}, \qquad &d=1;\\
C_2  \nu^3, \qquad &d=2;\\
C_3 \nu^7, \qquad &d=3,
\end{cases}
\end{align*}
where $C_1$, $C_2$, $C_3>0$ are constants depending only on the initial energy 
$E_0=\mathcal E (u^0)$. 
\end{thm}
\begin{rem*}
By Proposition \ref{prop_Eu}, we have $E(u^0) \lesssim 1 +E(u_0)$ uniformly in $N$ and $\nu$.
\end{rem*}
\begin{rem*}
Note that our $L^{\infty}$-assumption is made on $u^0$ instead of $u_0$. 
For $d=1$, we have by Sobolev embedding $\|u^0\|_{\infty} \lesssim
\| u_0 \|_{H^1} <\infty$. For dimension $d=2$ and $d=3$, one should note that the
mere assumption $u_0 \in L^{\infty}$ in general does not guarantee $u^0 \in L^{\infty}$
since the spectral projection is a non-smooth cut-off in frequency space.

\end{rem*}

\begin{rem}
The constants $C_j$, $1\le j\le 3$ in Theorem \ref{thm1} can be quantified explicitly. 
See in particular Theorems \ref{thmnu1a}, \ref{thmnu1b},
\ref{thm2dnu}, \ref{thm1dnu} and \ref{thm3dnu} for more precise statements. Heuristically speaking 
in general the threshold 
time step $\tau_{\operatorname{max}}$ can be determined via the relation:
\begin{align*}
   \sqrt{\frac {2\nu} {\tau_{\operatorname{max}} }} = \frac 32 L^2\Rightarrow
   \tau_{\operatorname{max}}= \frac {8} 9 \nu \cdot L^{-4},
\end{align*}
where 
\begin{align*}   
 L=  \max_{n\ge 0} \| u^n \|_{\infty}. 
\end{align*}
 In 2D the $L^{\infty}$-norm of $u^n$ can almost be bounded by the $H^1$ norm
 of $u^n$ which in turn is bounded by $\nu^{-\frac 12} \sqrt{2E(u^n)}$
 multiplied by some logarithm factors depending on $\nu$. However in our analysis we shall
 remove this logarithm and obtain the optimal scaling. 
 For general $d$,  we have the heuristic bound (below we neglect the dependence of the constants
 on energy and focus only on the $\nu$-dependence)
 \begin{align*}
 L \lesssim \| P_{\lesssim (\nu \tau_{\operatorname{max}} )^{-\frac 14} }  \|_{L^4 \to L^{\infty}} 
 \lesssim (\nu \tau_{\operatorname{max}} )^{-\frac{d}{16}},
 \end{align*}
where $P_{\lesssim (\nu \tau_{\operatorname{max}} )^{-\frac 14} } $ is 
a frequency localization operator.  From this one can roughly
determine $\tau_{\operatorname{max}}$ as
\begin{align*}
\tau_{\operatorname{max}} \sim \nu^{\frac {4+d} {4-d}}.
\end{align*}
\end{rem}
Our proof of Theorem \ref{thm1} proceeds by a simple yet powerful 
Trade-Energy-For-$L^{\infty}$ (TEFL) scheme which is a refinement of our earlier work
\cite{LQT17, LQ17, LQ173d}.  In several cases we even manage to calculate explicit constants 
and identified nearly optimal parametric dependences which seem to be the first done
in the literature. These  will be instrumental for future refined analysis on these algorithms.
It is expected that this new streamlined proof can be adapted to higher order cases
and generalized to many other models and settings.

For the first order IMEX scheme \eqref{sem1}, our TEFL recipe consists of three steps.

Step 1. Discrete energy estimate. We show that
\begin{align} \notag 
 &\mathcal E(u^{n+1}) -\mathcal E(u^n) 
 +\left(\frac 12+\sqrt{\frac{2\nu}{\tau}} \right) \|u^{n+1}-u^n\|_2^2 
 \le 
  \|u^{n+1}-u^n \|_2^2 \cdot \frac 32 \operatorname{max}
   \{ \|u^n\|_{\infty}^2, \;
 \| u^{n+1} \|_{\infty}^2 \}.
 \end{align}
Thus to show energy monotonicity it suffices to show
\begin{align*}
\sqrt{\frac{2\nu}{\tau}} \ge \frac 32 \operatorname{max}
   \{ \|u^n\|_{\infty}^2, \;
 \| u^{n+1} \|_{\infty}^2 \}.
\end{align*}

Step 2. Trade energy for $L^{\infty}$. This is the key step. We split $u^{n+1}$ as
\begin{align*}
u^{n+1}= \mathcal L_1 u^n + \mathcal L_2 \Bigl(\mathcal N(u^n) \Bigr),
\end{align*}
where $\mathcal L_1$ and $\mathcal L_2$ are both  linear operators mimicking the resolvent
of elliptic type operators and $\mathcal N(u^n)$
denotes the nonlinear part.  We then prove a direct $L^{\infty}$ estimate using only the energy
conservation and certain smoothing properties of the operators $\mathcal L_1$ and $\mathcal L_2$.
To achieve an ``optimal trade" (i.e. optimal dependence on $\nu$) it is of some importance
to use scaling-critical norms. In the end one obtains
\begin{align*}
 \|u^{n+1}\|_{\infty} \le C_{E(u^n)} \cdot h(\nu,\tau),
 \end{align*}
 where $C_{E(u^n)}>0$ depends only on the energy $E(u^n)$, and $h(\nu, \tau)$
 typically has the form $h(\nu,\tau)=\nu^{-\alpha} \tau^{-\beta}$ for some exponents
 $\alpha>0$, $\beta>0$. 
 
 Step 3. Identification of the optimal time step constraint. Here we work on the inequality
 \begin{align*}
\sqrt{\frac{2\nu}{\tau}} \ge \frac 32 (C_{E(u^n)} \cdot h(\nu,\tau) )^2
\end{align*}
and determine an optimal $\tau_{\operatorname{max}}=\tau_{\operatorname{max}}
(\nu, E)$. A suitable induction procedure then closes the needed estimates and
 yields the result.

The rest of this paper is organized as follows. In Section $2$ we set up the notation and collect
various preliminary materials. The discrete energy inequality is proved in Lemma \eqref{z2}.
Some Sobolev inequalities on $\mathbb T^2$ with explicit constants are also presented here.
Section $3$ and $4$ are devoted to the case $\nu=1$ with slightly different two approaches.
We showcase the proofs with explicit constants.  In Section $5$ we present the our streamlined
TEFL proof for the general case $\nu>0$ in dimension two. In Section $6$ we explain
the modifications needed for dimensions $1$ and $3$ respectively. In the last section
we give concluding remarks.

\section{Notation and preliminaries}
For any two positive quantities $X$ and $Y$, we shall write $X\lesssim Y$ or $Y\gtrsim X$ if
$X \le  CY$ for some  constant $C>0$ whose precise value is unimportant. 
We shall write $X\sim Y$ if both $X\lesssim Y$ and $Y\lesssim X$ hold.
We write $X\lesssim_{\alpha}Y$ if the
constant $C$ depends on some parameter $\alpha$. 
We shall 
write $X=O(Y)$ if $|X| \lesssim Y$ and $X=O_{\alpha}(Y)$ if $|X| \lesssim_{\alpha} Y$. 

We shall denote $X\ll Y$ if
$X \le c Y$ for some sufficiently small constant $c$. The smallness of the constant $c$ is
usually clear from the context. The notation $X\gg Y$ is similarly defined. Note that
our use of $\ll$ and $\gg$ here is \emph{different} from the usual Vinogradov notation
in number theory or asymptotic analysis.

For any $x=(x_1,\cdots, x_d) \in \mathbb R^d$, we denote
\begin{align*}
|x| =|x|_2=\sqrt{x_1^2+\cdots+x_d^2}, \quad
|x|_{\infty} =\max_{1\le j \le d} |x_j|.
\end{align*}
Also occasionally we use the Japanese bracket notation:
\begin{align*}
\langle x \rangle =(1+|x|^2)^{\frac 12}.
\end{align*}

For $1\le p \le \infty$ and any function $f:\, x\in \mathbb T^d \to \mathbb R$, we denote
the Lebesgue $L^p$-norm of $f$ as
\begin{align*}
\|f \|_{L^p_x(\mathbb T^d)} =\|f\|_{L^p(\mathbb R^d)} =\| f \|_p. 
\end{align*}
If $(a_j)_{j \in I}$ is a sequence of complex numbers 
and $I$ is the index set, we denote the discrete $l^p$-norm
as
\begin{align*}
\| (a_j) \|_{l_j^p(j\in I)} = \| (a_j) \|_{l^p(I)} =
\begin{cases}
 (\sum_{j\in I} |a_j|^p)^{\frac 1p}, \quad 0<p<\infty, \\
 \sup_{j\in I} |a_j|, \quad p=\infty.
 \end{cases}
 \end{align*}
 For example
 \begin{align*}
 \| \hat f(k) \|_{l_k^2(\mathbb Z^d)} = (\sum_{k \in \mathbb Z^d} |\hat f(k)|^2)^{\frac 12}.
 \end{align*}

If $f=(f_1,\cdots,f_m)$ is a vector-valued function, we denote
$|f| =\sqrt{\sum_{j=1}^m |f_j|^2}$, and
\begin{align*}
\| f\|_p = \|\, |f |\, \|_p =\| ({\sum_{j=1}^m f_j^2})^{\frac 12} \|_p.
\end{align*}
We use similar convention for the corresponding discrete $l^p$ norms for the vector-valued
case.

We denote
\begin{align*}
\operatorname{sgn}(x) =\begin{cases}
1, \quad x>0;\\
-1, \quad x<0.
\end{cases}
\end{align*}

We use the following convention for the Fourier transform pair:
\begin{align*}
&\hat f(k) = \int_{\mathbb T^d} f(x) e^{-2\pi i k\cdot x} dx, \quad
 f(x) =\sum_{k\in \mathbb Z^d} \hat f(k) e^{2\pi ik \cdot x}.
\end{align*}

We denote for $0\le s \in \mathbb R$, 
\begin{align*}
&\|f \|_{\dot H^s} = \|f \|_{\dot H^s(\mathbb T^d)} = \| |\nabla|^s f \|_{L^2(\mathbb T^d)}
= \|  (2\pi |k|)^s \hat f (k) \|_{l^2_k (\mathbb Z^d)}, \\
& \| f \|_{H^s} = \sqrt{\| f \|_2^2 + \| f\|_{\dot H^s}^2} = \| \langle 
2 \pi |k| \rangle^s \hat f(k) \|_{l^2_k(\mathbb Z^d)}.
\end{align*}

To simplify the notation, in the later part of this paper we shall often denote
\begin{align*}
E_n = \mathcal E (u^n)= \int_{\mathbb T^d}
( \frac 12 \nu |\nabla u^n|^2 + 
\frac 14 ( (u^n)^4 -2 (u^n)^2+1) ) dx,
\end{align*}
where $u^n$ is the discrete numerical solution computed according to the scheme
\eqref{sem1}. Note that $E_0 =\mathcal E (u^0) \ne \mathcal E(u_0)$ in general since
$u^0= \Pi_N u_0$. The next proposition clarifies this point.

\begin{prop}[Relation between $\mathcal E(\Pi_N f)$ and $\mathcal 
E(f)$] \label{prop_Eu}
Let $d\le 3$. 
For any $f \in H^1(\mathbb T^d)$, we have
\begin{align*}
& \operatorname{1)} \; \sup_{N\ge 1} \mathcal E (\Pi_N f) \le
\beta_1  \mathcal E(f)+\beta_2;\\
& \operatorname{2)} \; \lim_{N\to \infty}  \mathcal E(\Pi_N f) = \mathcal E(f).
\end{align*}
Here $\beta_1>0$, $\beta_2>0$ are constants depending only on $d$. 
\end{prop}
\begin{rem*}
One may wonder whether it is possible to get rid of $\beta_2$ and prove a perfect
inequality of the form
\begin{align} \label{perfect1a}
\sup_{N\ge 2} \int ( (\Pi_N f)^2-1)^2 dx \lesssim \int (f^2-1)^2 dx.
\end{align}
This is in general not valid. For simplicity consider 1D and $\mathbb T=[-\frac 12, \, \frac 12)$.
Note that for $f=\operatorname{sgn}(x)$, $\Pi_2 f$ or $\Pi_{N_0} f$ for any finite $N_0$
clearly does not vanish.  One can then 
take a suitable mollification $f^{\epsilon}$ of $f =\operatorname{sgn}(x)$ 
to disprove the inequality in this case. By taking $\nu>0$ sufficiently small,
one can then show $\mathcal E (\Pi_2 f^{\epsilon}) \gg \mathcal E( f^{\epsilon})$. 
\end{rem*}

\begin{proof}
We first note that $\Pi_N$ can be expressed as the product of one-dimensional Hilbert-type
 transforms and
 \begin{align*}
 \sup_{N\ge 2} \| \Pi_N f \|_{L^4(\mathbb T^d)} \le\; c_1\|f \|_{L^4(\mathbb T^d)},
 \end{align*}
where $c_1>0$ is some constant depending only on $d$. Clearly then
\begin{align*}
\int_{\mathbb T^d}( (\Pi_N f)^4 - 2 (\Pi_N f)^2+1) dx
& \le \int_{\mathbb T^d} (  c_1^4 |f|^4 +1) dx \notag \\
& \le \int_{\mathbb T^d}  ( 2c_1^4 (f^2-1)^2 +2c_1^4 +1) dx.
\end{align*}
Since $\| \nabla \Pi_N f\|_2 \le \| \nabla f \|_2$, it follows easily that 1) holds.

2) Since $\lim_{N\to \infty} \| \nabla \Pi_N f\|_2 =\| \nabla f\|_2$, we only need to
check the double well energy part. Now denoting
$\Pi_{>N}= \operatorname{Id}- \Pi_N$, we have
\begin{align*}
 & | \int ( (\Pi_N f)^2-1)^2 dx - \int ( f^2-1)^2 dx | \notag \\
 \lesssim& \| \Pi_N f - f\|_4
 \cdot ( \| \Pi_N f\|_4 + \|f\|_4) \cdot ( \| \Pi_N f\|_4^2 + \| f\|_4^2 +1) \notag \\
 \lesssim & \| \Pi_{>N} f \|_{H^1} \cdot ( \|f\|_4+ \| f\|_{4}^3) \to 0,
 \quad \text{as $N\to \infty$}.
 \end{align*}
 Thus 2) holds.

\end{proof}

\begin{lem} \label{m001}
Let $0\le \epsilon<1$. 
For any $\omega \in \mathbb R^2$ with $|\omega|=1$, we have
\begin{align*}
F(\omega)=\int_{ [-\frac 12,\frac 12)^2}
(1+\epsilon^2 |k|^2 +2 \epsilon \omega \cdot k)^{-1} dk \ge 
1+ \frac 16 \epsilon^2 -\frac 7 {30} \epsilon^4.
\end{align*}
In particular if $\epsilon^2 \le \frac 57$, then $F(\omega)\ge 1$.
\end{lem}
\begin{proof}
Denote the integrand as $g(k)$. By symmetry we have 
\begin{align*}
F(\omega)= \int_{[-\frac 12, \frac 12)^2} \frac 12 (g(k) +g(-k) ) dk.
\end{align*}
Denote $a=\epsilon^2$. We shall slightly abuse the notation and write $k^2=|k|^2$.
Clearly (note that $|k|^2\le \frac 12$)
\begin{align*}
\frac 12 (g(k)+g(-k)) &= (1+ak^2)^{-1} \cdot( 1- \frac {4a (\omega\cdot k)^2}{(1+ak^2)^2} )^{-1}
\notag \\
& \ge (1+a k^2)^{-1}
\cdot ( 1+ \frac {4a (\omega \cdot k)^2} {(1+ak^2)^2})
 \notag \\
& \ge 1-a k^2 + (1-ak^2)^3 \cdot 4a (\omega\cdot k)^2.
\end{align*}
By using symmetry (under the swapping of variables $k_1\leftrightarrow k_2$)
and the fact that $|\omega|=1$, we have for any $f$,
\begin{align*}
\int_{[-\frac 12, \frac 12)^2} f(k^2) (\omega \cdot k)^2 dk_1 dk_2= \frac 12 
\int_{[-\frac 12, \frac 12)^2} f(k^2) k^2 dk_1 dk_2.
\end{align*}
Here we also used the fact 
\begin{align*}
\int_{[-\frac 12, \, \frac 12)^2} f(k^2) k_1 k_2 dk_1 dk_2 =0.
\end{align*}
Thus we only need to work with the integrand
\begin{align*}
1-a k^2 + (1-ak^2)^3 \cdot 2a k^2.
\end{align*}
An explicit calculation then yields the result
\begin{align*}
1+ \frac a 6 -
\frac {7a^2}{30}
+ \frac {9a^3}{140} - 
\frac {83a^4}{12600} \ge 1+ \frac a 6 -
\frac {7a^2}{30}.
\end{align*}

\end{proof}
\begin{rem*}
There is some subtle dependence of the parameters  when we consider
the general inequality $$
\int_{[-\frac 12,\frac 12)^2} (1+\epsilon^2 |k|^2+2 \epsilon
\omega \cdot k)^{-s} dk >1 \quad \text{or $\le 1$} 
$$ 
for $s>0$ and $0<\epsilon \ll 1$. 
Note that for $0<\epsilon\ll 1$, we have (below  $a=\epsilon^2$, $r=a |k|^2=
\epsilon^2 |k|^2$)
\begin{align*}
& X=  {4 a (\omega \cdot k )^2}  (1+r)^{-2} 
= 4 a (\omega \cdot k)^2 (1-2 r +3 r^2) + O(a^4); \\
& \ln(1+r)+\frac 12 \ln (1-X)
= r-\frac 12 r^2 +\frac 12 (-X - \frac 12 X^2) + O(a^3).
\end{align*} 
By an explicit computation, we have
\begin{align*}
&\int_{[-\frac 12, \frac 12)^2}
( r - \frac 12 r^2 - \frac 12 X) dk_1 dk_2= \frac 7 {120} a^2 +O(a^3);\\
&\int_{[-\frac 12, \frac 12)^2}
(-\frac 14 X^2) dk_1 dk_2 =-4 a^2\int_{[-\frac 12, \frac 12)^2}
(\omega\cdot k)^4 dk_1 dk_2 +O(a^3) \notag \\
&\qquad\qquad\qquad\qquad\qquad=
-\frac {a^2}{60} (3 \omega_1^4+ 10 \omega_1^2\omega_2^2+
3 \omega_2^4) + O(a^3) \notag \\
&\qquad\qquad\qquad\qquad\qquad=
-\frac {a^2}{60} (3 +4 \omega_1^2\omega_2^2) + O(a^3),
\end{align*}
where in the last equality we used the fact that $\omega_1^2+\omega_2^2=1$.

Then  clearly
\begin{align*}
H(\epsilon,\omega)= &\int_{[-\frac 12,\frac 12)^2} \ln (1+\epsilon^2 |k|^2+2 \epsilon
\omega \cdot k) dk  \notag \\
=& \int_{[-\frac 12, \frac 12)^2}
( \ln (1+\epsilon^2 |k|^2) + \frac 12 \ln (1-\frac{4\epsilon^2 (\omega\cdot k)^2}
{(1+\epsilon^2 |k|^2)^2}) )dk\notag \\
=& \int_{[-\frac 12, \frac 12)^2} (r-\frac 12 r^2 +\frac 12 (-X - \frac 12 X^2) )
dk  + O(a^3) \notag \\
=& \frac 1 {60} a^2 \Bigl(0.5-4 \omega_1^2\omega_2^2)  \Bigr)+O(a^3) 
\end{align*}
Note that in the above calculation in the main order the integral is dependent on $\omega$. In particular its sign is dependent on the choice of $\omega$. 
If $\omega= \omega^{(1)}= (1,0)$, then apparently
\begin{align*}
H( \epsilon, \omega^{(1)} ) = \frac 1{120} a^2 +O(a^3)
\ge \frac 1 {200} \cdot \epsilon^4>0.
\end{align*}
If $\omega=\omega^{(2)}=(\frac 1{\sqrt 2}, \frac 1{\sqrt 2})$, then
\begin{align*}
H(\epsilon, \omega^{(2)}) = -\frac 1{120} a^2 +O(a^3)
\le - \frac 1 {200} \cdot \epsilon^4<0.
\end{align*}

 Now denote for $s>0$, 
$0<\epsilon \le 1/4$, 
\begin{align*}
I(s,\omega)= &\int_{[-\frac 12,\frac 12)^2} e^{-s \ln (1+\epsilon^2 |k|^2+2 \epsilon
\omega \cdot k)} dk  
= \int_{[-\frac 12,\frac 12)^2} (1+\epsilon^2 |k|^2+2 \epsilon
\omega \cdot k)^{-s} dk.
\end{align*}
Obviously
\begin{align*}
(\partial_s I)(0,\omega)= -H(\epsilon, \omega).
\end{align*}
It is not difficult to check that for some absolute constant $C_1>0$ we have $
\sup_{\omega} |\partial_{ss} I(s, \omega)| \le C_1$
for all $0<s\le 1/4$, $0<\epsilon\le 1/4$. It follows that  for some $s_0(\epsilon)$ depending
on $\epsilon$, we have
\begin{align*}
|\partial_s I(s, \omega^{(j)} )
- \partial_s I(0,\omega^{(j)}) | \le  \frac 1{10} |\partial_s I(0, 
\omega^{(j)} )|=\frac 1{10} |H(\epsilon, \omega^{(j)})|, \quad j=1,2,\; \text{for all $0\le s \le s_0(\epsilon)$}.
\end{align*}
Consequently 
\begin{align*}
&I(s, \omega^{(1)} ) \le 1 - \frac 1 {400} s \epsilon^4;\\
&I(s, \omega^{(2)} ) \ge 1 + \frac 1 {400} s \epsilon^4,
\end{align*}
 for $0<\epsilon \ll 1$ and $
0\le s \le s_0(\epsilon)$.  In other words one should take $s$ moderately large in order to
exceed $1$. Alternatively for small $\epsilon$ one can directly expand the integrand
$(1+\epsilon^2 |k|^2+2 \epsilon
\omega \cdot k)^{-s} $ in binomial series
and obtain
\begin{align*}
I(s, \omega^{(1)} )= 1+\frac 1 {360} \epsilon^2 s (60 s + \epsilon^2 (-3-2s +4s^2 +3s^3) )+ O(\epsilon^6).
\end{align*}
Clearly then $I(s, \omega^{(1)} ) <1$ if $s \ll \epsilon^2$. Similarly
\begin{align*}
I(s, \omega^{(2)})= 1+
\frac 1 {360} \epsilon^2 s 
(60 s+ \epsilon^2 (3+9s+10s^2+4s^3) ) + O(\epsilon^6).
\end{align*}
Clearly $I(s,\omega^{(2)})>1$ for $0<\epsilon \ll 1$.
\end{rem*}
\begin{rem*}
Note that $\ln | \cdot |$ is harmonic  and $|\cdot |^{-s}$ ($s>0$)
is subharmonic on $\mathbb R^2\setminus\{0\}$. The preceding computations show that
a subharmonic approach is not suitable. This is hardly surprising since we are working with a square.
\end{rem*}

\begin{lem} \label{m003}
Let 
\begin{align*}
I(n)= \int_{ [-\frac 12,\frac 12)^2} \frac 1 {|k+n|^2} dk - \frac 1 {|n|^2}.
\end{align*}
By symmetry, we have for any $n=(n_1,n_2)$, $I(n)=I(\pm n_1, \pm n_2)$. 
Furthermore
\begin{align*}
&I(1,0)\ge 0.1731, \quad
I(1,1) \ge 0.0525,\\
&I(2,0) \ge 0.0105,\quad
I(2,1) \ge 0.007, \quad I(2,2)\ge 0.0027\\
&I(3,0) \ge 0.0020,\quad I(3,1)\ge 0.0016,\quad I(3,2)\ge 0.0010,\quad
I(3,3)\ge 0.0005.
\end{align*}

\end{lem}
\begin{proof}
Direct numerical verification.
\end{proof}

\begin{lem} \label{m005}
For any $\frac 1 {\sqrt 2} <r_1 <r_2$, we have
\begin{align*}
\sum_{\substack{r_1\le |n|\le r_2 \\ n \in \mathbb Z^2}}
\frac 1 {|n|^2}
\le 2\pi  \Bigl(\ln (r_2+ \frac 1 {\sqrt 2}) -  \ln (r_1-\frac 1{\sqrt 2} ) \Bigr).
\end{align*}
Also for $r_2>\frac 1{\sqrt 2}$, 
\begin{align*}
\sum_{\substack{ |n|\ge r_2\\
n \in \mathbb Z^2} }
\frac 1 {|n|^4} \le \pi (r_2- \frac 1 {\sqrt 2} )^{-2},
\end{align*}
and more generally for any $s>2$, 
\begin{align*}
\sum_{\substack{ |n|\ge r_2\\
n \in \mathbb Z^2} }
\frac 1 {|n|^s} \le 2\pi\cdot \frac 1 {s-2} (r_2- \frac 1 {\sqrt 2} )^{-(s-2)}.
\end{align*}
\end{lem}

\begin{rem*}
Consequently for any $r\ge 2$, we have
\begin{align}
\sum_{\substack{0<|k|\le r\\ k \in \mathbb Z^2}} |k|^{-2}
&= 6+ \sum_{2\le |k|\le r} |k|^{-2} \le 6 + 2\pi
( \ln(r+\frac 1{\sqrt 2}) - \ln (2-\frac 1{\sqrt 2})) \notag \\
& \le 4.4 + 2\pi \ln (r+\frac 1 {\sqrt 2}). \label{eA001}
\end{align}
\end{rem*}

\begin{proof}
By Lemma  \ref{m001} and \ref{m003} we have
\begin{align*}
\sum_{r_1\le |n|\le r_2}
|n|^{-2}
\le \sum_{r_1\le |n|\le r_2} \int_{ [-\frac 12, \frac 12)^2}
|n+k|^{-2} dk.
\end{align*}
Now for each $n \in \mathbb Z^2$, define $I_n=\{z\in \mathbb R^2:\; |z-n|_{\infty} <
\frac 12\}$. Clearly $I_n$ and $I_{n^{\prime}}$ are disjoint if $n\ne n^{\prime}$. 
Obviously
\begin{align*}
\bigcup_{r_1 \le |n|\le r_2} I_n \subset 
\{ z\in \mathbb R^2:\;  r_1-\frac 1{\sqrt 2}
\le |z| \le r_2+\frac 1 {\sqrt 2}  \}.
\end{align*}
Thus
\begin{align*}
\sum_{r_1\le |n|\le r_2}
|n|^{-2}
\le \int_{  r_1-\frac 1{\sqrt 2}
\le |z| \le r_2+\frac 1 {\sqrt 2} } |z|^{-2} dz
=2\pi  \Bigl(\ln (r_2+ \frac 1 {\sqrt 2}) -  \ln (r_1-\frac 1{\sqrt 2} ) \Bigr).
\end{align*}
For the second inequality, we note that 
\begin{align*}
|n|^{-4}= (|n|^{-2})^2 \le 
(\int_{[-\frac 12, \frac 12)^2} |n+k|^{-2} dk )^2 \le 
\int_{[-\frac 12, \frac 12)^2} |n+k|^{-4} dk.
\end{align*}
Thus
\begin{align*}
\sum_{|n|\ge r_2} |n|^{-4} \le \int_{|z|\ge r_2-\frac 1 {\sqrt 2}} |z|^{-4} dz
= \pi (r_2- \frac 1 {\sqrt 2} )^{-2}.
\end{align*}

\end{proof}

\begin{lem} \label{m007}
Let $f \in H^2(\mathbb T^2)$ with mean zero. Then 
for any $r\ge  2$, we have
\begin{align*}
\|f \|_{L^{\infty}(\mathbb T^2)}
\le \| f \|_{\dot H^1} \cdot  \frac 1 {2\pi}
\cdot (4.386 + 2\pi( \ln(r+\frac 1 {\sqrt 2}) ) )^{\frac 12}
 + \|f \|_{\dot H^2} \cdot \frac 1 {4\pi^{\frac 32}} (r-\frac 1 {\sqrt 2})^{-1}.
\end{align*}

\end{lem}
\begin{proof}
Splitting $f$ into low and high frequencies,  we have
\begin{align*}
\|f \|_{\infty} \le 
\| f \|_{\dot H^1} \cdot ( \sum_{0<|k| \le r}
\frac 1 {(2\pi |k|)^2} )^{\frac 12}
+ \| f\|_{\dot H^2}
\cdot ( \sum_{|k|>r}
\frac 1 {(2\pi |k|)^4} )^{\frac 12}.
\end{align*}
By Lemma \ref{m005}, we have
\begin{align*}
&\sum_{0<|k|\le r} |k|^{-2}= 6 + \sum_{2 \le |k| \le r}
|k|^{-2} \le  6 +2\pi ( \ln (r+\frac 1 {\sqrt{2}}) -\ln (2-\frac 1{\sqrt 2}) )
\le 4.386+ 2\pi \ln (r+ \frac 1 {\sqrt 2} )
, \\
&\qquad \sum_{|k|>r} |k|^{-4} \le \pi  (r-\frac 1 {\sqrt 2})^{-2}.
\end{align*}
\end{proof}

\begin{lem} \label{m009}
Denote $\mathbb T$ the one-periodic torus on $\mathbb R$ which we identify
as $\mathbb T=[-\frac 12, \frac 12)$. 
For any smooth $f:\; \mathbb T \to \mathbb R$, we have
\begin{align*}
\| f -\bar f \|_{L^{\infty}(\mathbb T)} \le \frac 12 \| f^{\prime} \|_{L^1(\mathbb T)},
\end{align*}
where $\bar f$ denotes the average of $f$ on $\mathbb T$. 
On $\mathbb T^2=[-\frac 12, \frac 12)^2$, we have
\begin{align*}
\|f -\bar f\|_{L^2(\mathbb T^2)}
&\le \frac 12 \| (|\partial_1 f | + |\partial_2 f |) \|_{L^1(\mathbb T^2)} \notag \\
&  \le \frac 1 {\sqrt 2}
\| \nabla f \|_{L^1(\mathbb T^2)} = \frac 1{\sqrt 2}\|  \sqrt{ |\partial_1 f|^2 +|\partial_2 f|^2}
\|_{L^1(\mathbb T^2)}.
\end{align*}
Note that we use the convention $|\nabla f | =\sqrt{ |\partial_1 f|^2 +|\partial_2 f|^2}$.
\end{lem}
\begin{rem*}
The first inequality can be achieved by a smooth approximation of
$f= \frac 12 \operatorname{sgn}(x)$ on $[-\frac 12, \frac 12)$. The constant in the second inequality is not sharp. 
\end{rem*}
\begin{proof}
The first inequality can be proved in two ways. Without loss of generality one may assume
$\bar f=0$. One can then choose some $x_0 \in \mathbb T$ such that $f(x_0)=0$. Writing
$f(x)=\int_{x_0}^x f^{\prime}(s) ds = -\int_x^{x_0+1} f^{\prime}(s) ds$ then yields the
result. Alternatively one can resort to Fourier analysis and show that 
\begin{align*}
K(x)=\mathcal F^{-1} ( \frac 1 {2\pi i k} 1_{k \ne 0}) =\frac 12 \operatorname{sgn}(x)
-x =\begin{cases}
\frac 12 -x, \quad 0<x \le \frac 12;\\
-\frac 12-x, \quad -\frac 12\le x<0.
\end{cases}
\end{align*}
Obviously $\|K\|_{\infty} \le \frac 12$ and the desired inequality follows.

For the second inequality we may also assume $\bar f=0$. Observe
\begin{align*}
&|f(x_1, x_2) - \underbrace{\int f(y_1, x_2) dy_1}_{=:B(x_2)}| \le \frac 12 \int |\partial_1 f(y_1, x_2) | dy_1, \\
&|f(x_1,x_2)- \underbrace{\int f(x_1, y_2) dy_2}_{=:A(x_1)} | \le \frac 12 \int | \partial_2 f(x_1, y_2)| dy_2.
\end{align*}
Note that  since $\bar f =0$ we have $\int A(x_1) dx_1= 0= \int B(x_2) dx_2$. Clearly then
\begin{align*}
& \| A(x_1) \|_{L_{x_1}^{\infty}} \le \frac 12 \int | \partial_1 A| dx_1  \le \frac 12\| \partial_1 f\|_{L^1},\quad
 \| B(x_2) \|_{L_{x_2}^{\infty} } \le \frac 12\| \partial_2 f \|_{L^1}.
\end{align*}
It follows that 
\begin{align*}
f^2 &\le -A(x_1) B(x_2)+ (A(x_1)+B(x_2)) f + \frac 14 \int |\partial_1 f(y_1, x_2) | dy_1 \cdot
\int | \partial_2 f(x_1, y_2)| dy_2 \notag \\
& \le \frac 1 2 f^2 + \frac 12 (A(x_1) +B(x_2))^2 
-A(x_1) B(x_2)+\frac 14 \int |\partial_1 f(y_1, x_2) | dy_1 \cdot
\int | \partial_2 f(x_1, y_2)| dy_2.
\end{align*}
We then obtain
\begin{align*}
\|f \|_{L^2}^2 & \le  \int A(x_1)^2 dx_1 + \int B(x_2)^2 dx_2 +
\frac 12 \| \partial_1 f \|_{L^1} \| \partial_2 f \|_{L^1} \notag \\
& \le \frac 14 ( \| \partial_1 f\|_{L^1} + \| \partial_2 f \|_{L^1} )^2,
\end{align*}
where in the last inequality we used $\| A(x_1)\|_{\infty} \le 
\frac 12 \| \partial_1 f \|_{L^1}$, $\|B(x_2) \|_{\infty} \le 
\frac 12 \| \partial_2 f \|_{L^1}$. 

\end{proof}

\begin{lem} \label{m011}

\begin{align*}
\|f -\bar f\|_{L^{\infty}(\mathbb T^2)}  \le \frac {\sqrt{6.05}}{(2\pi)^2}\|f \|_{\dot H^2(\mathbb T^2)}
\cdot 
\end{align*}
\end{lem}
\begin{proof}
Note that by Lemma \ref{m005}, for any $r>\frac 1 {\sqrt 2}$, we have 
\begin{align*}
\sum_{0\ne k \in \mathbb Z^2}
\frac 1 {(2\pi |k|)^4}\le 
\frac 1 {(2\pi)^4} ( \sum_{0<|k|\le r }\frac 1 {|k|^4}
+\pi (r-\frac 1{\sqrt 2})^{-2}).
\end{align*}
Choosing $r=10$ then yields the result. One should note that
\begin{align*}
&\sum_{0<|k|\le 10} |k|^{-4} \le 
\sum_{0<|k|_{\infty} \le 10} |k|^{-4} \approx 6.00355<6.0036;\\
& \pi (10-\frac 1{\sqrt{2}})^{-2} \approx 0.0363788<0.037.
\end{align*}
\end{proof}

\begin{lem}[Discrete energy estimate] \label{z2}
For any $n\ge 0$,
\begin{align}
 &E_{n+1}-E_n 
 +\left(\frac 12+\sqrt{\frac{2\nu}{\tau}} \right) \|u^{n+1}-u^n\|_2^2 
 \le 
  \|u^{n+1}-u^n \|_2^2 \cdot \frac 32 \operatorname{max}
   \{ \|u^n\|_{\infty}^2, \;
 \| u^{n+1} \|_{\infty}^2 \}.
 \end{align}

\end{lem}

\begin{proof}
In this proof we denote by $(\cdot,\cdot)$ the usual $L^2$ inner product.
Recall
\begin{align*}
\frac{u^{n+1}-u^n}{\tau} = -\nu \Delta^2 u^{n+1}+\Delta \Pi_N f(u^n).
\end{align*}
Taking the $L^2$ inner product with $(-\Delta)^{-1}(u^{n+1}-u^n)$ on both sides and 
applying the identity
\begin{align} \notag
b\cdot (b-a) = \frac 12 ( |b|^2-|a|^2+|b-a|^2), \qquad \forall\, a, b\in \mathbb R^d,
\end{align}
we get
\begin{align} \notag 
  & \frac 1 {\tau} \| |\nabla|^{-1} (u^{n+1}-u^n) \|_2^2
  + \frac {\nu}2 ( \| \nabla u^{n+1} \|_2^2 -\| \nabla u^n\|_2^2
+ \| \nabla(u^{n+1}-u^n ) \|_2^2 ) \notag \\
& \qquad = ( \Delta \Pi_N f(u^n), (-\Delta)^{-1} (u^{n+1}-u^n) ).
\end{align}

Since $u^n$ and $u^{n+1}$ have Fourier modes trapped in the sector
$\{k:\, |k|_{\infty}\le N \}$, we have
\begin{align} \notag
 (\Delta \Pi_N f(u^n), (-\Delta)^{-1}(u^{n+1}-u^n) ) = - (f(u^n), u^{n+1}-u^n).
 \end{align}

By using the auxiliary function $g(s)=F(u^n+s(u^{n+1}-u^n))$ (recall
$f=F^{\prime}$) and the Taylor expansion
\begin{align*}
g(1)=g(0)+g^{\prime}(0) + \int_0^1 g^{\prime\prime}(s) (1-s)ds,
\end{align*}
we get
\begin{align*}
F(u^{n+1}) &=F(u^n) +f(u^n) (u^{n+1}-u^n) -\frac 1 2 (u^{n+1}-u^n)^2 \notag \\
& \qquad + (u^{n+1}-u^n)^2 \int_0^1 \tilde f^{\prime}( u^n+s(u^{n+1}-u^n) )  (1-s)ds,
\end{align*}
where $\tilde f(z)=z^3$ and $\tilde f^{\prime}(z)=3z^2$ (for $z\in \mathbb R$). From this it is easy to see
that
\begin{align*}
 - (f(u^n), u^{n+1}-u^n) 
 \le F(u^n) - F(u^{n+1}) -\frac 1 2 \|u^{n+1}-u^n\|_2^2
+\|u^{n+1}-u^n\|_2^2 \cdot \frac 32
\max\{\|u^n\|_{\infty}^2,\|u^{n+1}\|_{\infty}^2\}.
\end{align*}

Thus
\begin{align} \notag 
  & E_{n+1}-E_n+ \frac 1 {\tau} \| |\nabla|^{-1} (u^{n+1}-u^n) \|_2^2
+\frac{\nu}2 \| \nabla(u^{n+1}-u^n ) \|_2^2  + \frac 12 \|u^{n+1}-u^n \|_2^2  \notag \\
&\qquad\qquad\qquad  \le \;\|u^{n+1}-u^n\|_2^2 \cdot \frac 32
\max\{\|u^n\|_{\infty}^2,\|u^{n+1}\|_{\infty}^2\}.
\end{align}

Finally observe
\begin{align*}
&\frac 1 {\tau} \| |\nabla|^{-1} (u^{n+1}-u^n) \|_2^2 + \frac{\nu}2 \| \nabla(u^{n+1}-u^n ) \|_2^2 \notag \\
\ge\; & \sqrt{\frac{2\nu} {\tau} } \| |\nabla|^{-1}(u^{n+1}-u^n) \|_2 \| \nabla (u^{n+1}-u^n) \|_2
\ge \sqrt{\frac{2\nu}{\tau} } \| u^{n+1}-u^n\|_2^2.
\end{align*}

The desired inequality then follows easily.

\end{proof}

\begin{lem} \label{leKbeta}
Let $d\le 3$ and $\beta>0$. Consider on the torus $\mathbb T^d$,
\begin{align*}
K(x) = \mathcal F^{-1} ( (1+\beta (2\pi |k|)^4)^{-1} ) 
=(1+\beta \Delta^2)^{-1} \delta_0,
\end{align*}
 where $\delta_0$ is the periodic Dirac comb. Then for any $1\le p\le \infty$, 
\begin{align*}
\| K \|_{L^p(\mathbb T^d)} \le c_{d,p}\,  (1+\beta^{-d(\frac 14 -\frac 1{4p})}),
\end{align*}
where $c_{d,p}>0$ depends only on $d$ and $p$.  Define
\begin{align*}
\widetilde{K}= \mathcal F^{-1} ( (1+\beta (2\pi |k|)^4)^{-1} 1_{k\ne 0} ).
\end{align*}
Then 
\begin{align*}
\| \widetilde{K} \|_{L^p(\mathbb T^d)} \le \tilde c_{d,p}\,  \beta^{-d(\frac 14 -\frac 1{4p})},
\end{align*}
where $\tilde c_{d,p}>0$ depends only on $d$ and $p$.  

\end{lem}
\begin{rem*}
\begin{align*}
\max_{1\le p \le \infty} (c_{d,p}+\tilde c_{d,p}) \le B_d<\infty,
\end{align*}
where $B_d$ depends only on $d$. 
\end{rem*}
\begin{rem*}
By examining the $L^2$ Fourier coefficients of $K$ one can see that the constant $1$
needs to be present in the $L^p$ upper bound. This also follows from the fact
that $\overline K= \int K =1$ and $\|K\|_{L^p(\mathbb T^d)}
\ge \| K\|_{L^1(\mathbb T^d)} \ge \overline K=1$. 
\end{rem*}

\begin{rem*}
For $\beta \ge 1$, one has the stronger bound on $\tilde K$ as
\begin{align*}
\| \tilde K\|_{L_x^1(\mathbb T^d)}  \le \| \tilde K \|_{L_x^{\infty}(\mathbb T^d)}  \lesssim_d \frac 1 {\beta}.
\end{align*}
By using the identity 
\begin{align*}
\frac {\beta (2\pi)^4} {1+ \beta (2\pi |k|)^4} = |k|^{-4} - \frac 1 {(2\pi)^4 \beta}
\cdot \frac 1 {|k|^4 (|k|^4+ \frac 1 {(2\pi)^4\beta})},
\end{align*}
we also have for $\beta\ge 1$, 
\begin{align*}
\| \tilde K \|_{L_x^1(\mathbb T^d)} \gtrsim_d \, \frac 1 {\beta}.
\end{align*}

\end{rem*}

\begin{proof}

Define
\begin{align*}
K_1(x) = \int_{\mathbb R^d} e^{-2\pi i \xi \cdot x} \cdot \frac 1 { 1+ (2\pi |\xi |)^4} d\xi.
\end{align*}
It is easy to check that $|K_1(x)| \lesssim \, \langle x \rangle^{-10}$ and $K_1 \in 
L_x^1(\mathbb R^d)$ for $d\le 3$. 

Now note that for $d\le 3$, if $|x|_{\infty} \le \frac 12$, then $|x| \le \sqrt d \cdot \frac 12
\le \frac {\sqrt 3}2$. Thus if $|l|\ge 4$, then 
\begin{align*}
&\frac 12 |l| \le |x+l| \le 2 |l|, \qquad \forall\,  |x|_{\infty} \le \frac 12.
\end{align*}
It follows that for all $1\le p\le \infty$ and $|l|\ge 4$,
\begin{align*}
\| \langle \beta^{-\frac 14} (x+ l) \rangle^{-10}
\|_{L_x^p(|x|_{\infty}<\frac 12)} \le 
\langle \beta^{-\frac 14} \frac 12 |l| \rangle^{-10}
\; \le \| \langle \beta^{-\frac 14} \frac 14 |x+ l| \rangle^{-10}
\|_{L_x^1(|x|_{\infty}<\frac 12)}.
\end{align*}

Clearly then
\begin{align}
\| K\|_{L^p(\mathbb T^d)} &\le \beta^{-\frac d 4} \sum_{l \in \mathbb Z^d}
\| K_1( \beta^{-\frac 14}( x+l) ) \|_{L_x^{p} (|x|_{\infty} <\frac 12)} \notag \\
& \lesssim \; 
\beta^{-\frac d4} 
\sum_{|l|\le 4} \| K_1 (\beta^{-\frac 14} (x+l)) \|_{p}
+ \beta^{-\frac d4} \sum_{|l|>4} 
\| \langle \beta^{-\frac 14} \frac 14 |x+ l| \rangle^{-10}
\|_{L_x^1(|x|_{\infty}<\frac 12)} \notag \\
& \lesssim \; \beta^{-d(\frac 14 - \frac 1 {4p} )} +1.  \label{tpK_e1}
\end{align}

Now we consider the estimate for $\widetilde{K}(x) = K(x)-1$. 
Obviously by using the previous bound we have $\|\widetilde{K}\|_1 \lesssim \|K\|_1+1
\lesssim 1$. Alternatively
one can compute
\begin{align*}
\| \tilde  K \|_{L_x^1(\mathbb T^d)} \le  1+
\| \sum_{l \in \mathbb Z^d} \beta^{-\frac d 4} |K_1 (\beta^{-\frac 14} (x+l) )|
\|_{L_x^1(\mathbb T^d)} \le  1+
\|  K_1 \|_{L_x^1(\mathbb R^d)} \lesssim 1.
\end{align*}
We then bound the $L^2$ norm as
\begin{align*}
\| \widetilde{K}\|_{L_x^2(\mathbb T^d)} = \| \frac 1 {1+ \beta (2\pi |k|)^4} \|_{l_k^2(
0\ne k \in \mathbb Z^d)}
\lesssim\,
\begin{cases}
\beta^{-1}, \quad \text{if $\beta\ge 1$};\\
\beta^{-\frac d8}, \quad \text{if $0<\beta<1$}.
\end{cases}
\end{align*}
Since $d\le 3$, we have the uniform bound for all $\beta>0$ as
\begin{align*}
\| \widetilde{K}\|_{L_x^2(\mathbb T^d)} \lesssim \; \beta^{-\frac d8}.
\end{align*}
Similarly
\begin{align*}
\| \widetilde{K}\|_{L_x^{\infty} (\mathbb T^d)} &\le  \| \frac 1 {1+ \beta (2\pi |k|)^4} \|_{l_k^1(
0\ne k \in \mathbb Z^d)}
\lesssim
\begin{cases}
\beta^{-1}, \quad \text{if $\beta\ge 1$};\\
\beta^{-\frac d4}, \quad \text{if $0<\beta<1$};
\end{cases} \notag\\
& \lesssim \beta^{-\frac d4}, \qquad\text{for all $\beta>0$ (since $d\le 3\Rightarrow \frac d4<1$)}.
\end{align*}
By using interpolation we then get the $L^p$ estimate. 
\end{proof}
\begin{rem*}
For $2\le p\le \infty$, one can also use
\begin{align*}
\| f \|_{L^p(\mathbb T^d)} \le \| \widehat f \|_{l_k^{\frac p {p-1}} (\mathbb Z^d)}
\end{align*}
which is a variant of Hausdoff-Young. Note the inequality itself also follows from interpolation.
\end{rem*}

\begin{rem*}
Interestingly it is also possible to bound $\| \tilde K\|_p$ directly on the real
side without using interpolation or Fourier transform by modifying
the argument in \eqref{tpK_e1}. Note that we only need to treat the case $\beta \gg 1$.
For simplicity consider 1D and $p=\infty$. Denote $\epsilon = \beta^{-\frac 14}$. It suffices
to check that
\begin{align*}
\| \Bigl(\sum_{n \in \mathbb Z} \epsilon K_1 (\epsilon (x+n ) ) \Bigr)-1\|_{L^{\infty}
(|x|\le \frac 12)} \lesssim\; \epsilon^4.
\end{align*}
To simplify we shall just consider the point $x=0$. The argument is similar and estimates
are uniform for all $|x|<\frac 12$.  Note that $K_1^{(4)} (x) = -K_1(x) +\delta(x)$. 
One can then apply the Euler-MacLaurin formula to the function $g(x)=
\epsilon K_1(\epsilon x)$. The tail term is then given by 
\begin{align*}
-\frac 1 {24} \int_{\mathbb R}   g^{(4)} (x) B_4( x-[x]) dx,
\end{align*}
where $B_4$ is the Bernoulli polynomial:
\begin{align*}
B_4(x) = x^2 (x-1)^2-\frac 1 {30},
\end{align*}
and $[x]$ denotes the smallest integer less than or equal to $x$. It is easy to check
that $B_4(x-[x])$ is continuous at the origin so that it can be paired with the Dirac delta function.
It follows easily that
\begin{align*}
| \int_{\mathbb R}   g^{(4)} (x) B_4( x-[x]) dx| \lesssim\; \epsilon^{4}.
\end{align*}
Note that one can also use truncation and 
mollification to make the whole argument rigorous. We omit
the details.
\end{rem*}

\section{Proof for 2D and  $\nu=1$: first approach}

Consider the semi-implicit scheme:
\begin{align} \label{semi_e1}
\frac{u^{n+1}-u^n}{\tau} = -   \Delta^2 u^{n+1} +\Delta \Pi_N ( f(u^n) ), \quad n\ge 0.
\end{align}
where $\tau>0$ is the time step. Then
\begin{align*}
(1+ \tau \Delta^2) u^{n+1} = u^n + \tau \Delta\Pi_N ( f(u^n)).
\end{align*}

\begin{lem} \label{m100}
Denote $\mathcal E (u^n)=E_n$. Then
\begin{align*}
& \|\Delta \frac {\tau \Delta}{1+\tau \Delta^2} \Pi_N ( (u^n)^3 - 3u^n)\|_{L^2(\mathbb T^2)}
\le  3  E_n\; ;\\
&\| u^{n+1} \|_{\dot H^2(\mathbb T^2)} \le  (2+\frac 1 {\tau} ) \sqrt {E_n}+ 3 E_n\; ;\\
&  \|u^{n+1}\|_{\dot H^1(\mathbb T^2)} \le  2\sqrt 2 \sqrt{E_n} +3E_n.
\end{align*}

\end{lem}
\begin{proof}
For simplicity of notation we denote $v=u^{n+1}$ and $u=u^n$. 
We shall prove the first two inequalities at one stroke.
Rewrite
\begin{align*}
v= \frac {1+2\tau \Delta}{1+\tau \Delta^2} u
+ \frac {\tau \Delta}{1+\tau \Delta^2}\Pi_N  ( u^3 - 3u ).
\end{align*}
The special splitting here is to make the bound for nonlinear part easier to express 
in terms of the energy.

Then by Lemma \ref{m009}, we have
(below $\overline{u^3}$ denotes the average of $u^3$ on $\mathbb T^2$)
\begin{align*}
\| v\|_{\dot H^2} & \le (2+\frac 1 {\tau})\|u\|_2 +
\| u^3-\overline{u^3}- 3u \|_2 \notag \\
& \le  (2+\frac 1 {\tau}) \frac 1 {\sqrt{2}} 
\| \nabla u \|_1 + \frac 1 {\sqrt 2} \| 3\nabla u (u^2-1) \|_1  \notag \\
& \le (2+\frac 1 {\tau}) \frac 1 {\sqrt{2}} \| \nabla u\|_2
+ 3 \| 2\cdot \frac 1 {\sqrt 2} \nabla u \cdot \frac 12 (u^2-1) \|_1 \notag \\
& \le (2+\frac 1 {\tau} ) \sqrt {E_n}+ 3 E_n.
\end{align*}

The third inequality is similarly proved. We omit details.
\end{proof}

\begin{thm}[Conditional energy stability for 2D $\nu=1$, first approach] \label{thmnu1a}
Let $d=2$, $\nu=1$ and $N\ge 2$.  Assume $u_0\in H^1(\mathbb T^2)$ and
has zero mean. Recall $u^0=\Pi_N u_0$ and assume $\|u^0\|_{\infty} \le 1.25$.
Take
\begin{align*}
 \tau_{\operatorname{max}} =
\begin{cases}
\frac 12,\quad \text{if $F_0\le 1$};\\
\min\{\frac 12,\, \frac 1 { (F_0 (1+\ln F_0))^2}\}, \quad\text{if $F_0>1$},
\end{cases}
\end{align*}
where $F_0= (2\sqrt{2E_0} +3E_0)^2$.
Then for any $0<\tau\le \tau_{\operatorname{max}}$, the scheme
\eqref{semi_e1} is energy stable, i.e.  
\begin{align*}
E(u^{n+1}) \le E(u^n), \qquad\forall\, n\ge 0.
\end{align*}

\end{thm}
\begin{rem*}
We chose the assumption $\|u^0\|_{\infty} \le 1.25$ for simplicity. It can be replaced
by a general upper $L_0$ and $\tau_{\operatorname{max}}$ can be adjusted accordingly with
some simple changes in numerology.
\end{rem*}

\begin{proof}
We use induction. By Lemma \ref{z2}, in order to have $E_{n+1} \le E_n$, the main condition
to verify is the inequality
\begin{align} \label{e500}
\frac 12 + \sqrt {\frac 2 {\tau}} \ge 
\frac 32 \max\{\| u^n \|_{\infty}^2, \,  
\| u^{n+1}\|_{\infty}^2\}.
\end{align}

Step 1. Base step $n=0$. Since by assumption $\|u^0\|_{\infty} \le 1.25$ and 
$0<\tau\le \frac 12$, it is clear that
\begin{align*}
\frac 12 + \sqrt{\frac {2}{\tau}} \ge 2.5 \ge \frac 32 \cdot 1.25^2 \ge 
\frac 32 \| u^0 \|_{\infty}^2.
\end{align*}
Thus to have $E_1\le E_0$ we only need to verify 
\begin{align} \label{e500a}
\frac 12 + \sqrt {\frac 2 {\tau}} \ge 
\frac 32 \| u^1\|_{\infty}^2.
\end{align}

We first note that if 
\begin{align*} 
(2+\frac 1 {\tau}) \sqrt{E_0} + 3E_0 \le \frac 83 \pi^{\frac 32},
\end{align*}
then by Lemma \ref{m011} and \ref{m100}, we have
\begin{align*}
 \|u^{1}\|_{\infty}^2 
 & \le \frac {6.05}{(2\pi)^4} ((2+\frac 1 {\tau}) \sqrt{E_0} + 3E_0 )^2 \notag \\
 & \le \frac {6.05}{(2\pi)^4}
\cdot \frac {64\pi^{3}} {9} = \frac {6.05}9 \cdot \frac 4 {\pi}\approx 0.8559.
\end{align*}
Since we assume $\tau \le \frac 12$, the inequality \eqref{e500a} then clearly holds
in this case. Thus in the following we may assume that we are given the condition
\begin{align} \label{e502}
(2+\frac 1 {\tau}) \sqrt{E_0} + 3E_0 > \frac 83 \pi^{\frac 32}.
\end{align}
We stress that  there is no need to deduce from 
\eqref{e502} any constraint on $\tau$.

By Lemma \ref{m007} and \ref{m100}, we have for any $r\ge  2$, 
\begin{align*}
\|u^{1} \|_{L^{\infty}(\mathbb T^2)}
&\le \| u^{1} \|_{\dot H^1} \cdot  \frac 1 {2\pi}
\cdot (4.386 + 2\pi( \ln(r+\frac 1 {\sqrt 2}) ) )^{\frac 12}
 + \|u^{1} \|_{\dot H^2} \cdot \frac 1 {4\pi^{\frac 32}} (r-\frac 1 {\sqrt 2})^{-1} \notag \\
 &\le (2\sqrt 2 \sqrt{E_0} +3E_0)
 \cdot  \frac 1 {2\pi}
\cdot (4.386 + 2\pi( \ln(r+\frac 1 {\sqrt 2}) ) )^{\frac 12} \notag \\
&\qquad + (  (2+\frac 1 {\tau} ) \sqrt {E_0}+ 3 E_0) \cdot \frac 1 {4\pi^{\frac 32}} (r-\frac 1 {\sqrt 2})^{-1}.
\end{align*}
Now we take
\begin{align*}
r= \frac 3 {4\pi^{\frac 32}}
( (2+\frac 1 {\tau}) \sqrt{E_0} +3 E_0 )+ \frac 1 {\sqrt 2} >2.
\qquad \text{ ( by \eqref{e502} ) }
\end{align*}
Then clearly 
\begin{align*}
(  (2+\frac 1 {\tau} ) \sqrt {E_0}+ 3 E_0) \cdot \frac 1 {4\pi^{\frac 32}} (r-\frac 1 {\sqrt 2})^{-1}
\le \frac 13
\end{align*}
We then only need to check the inequality 
\begin{align*}
\sqrt{ \frac 23 (\frac 12 +\sqrt{\frac 2{\tau}})}
\ge\; \frac 13 +(2\sqrt{2E_0} +3E_0)
 \cdot \frac 1 {2\pi}
\cdot (4.386 + 2\pi( \ln(r+\frac 1 {\sqrt 2}) ) )^{\frac 12}.
\end{align*}
By using the inequality $\sqrt{a+b} \ge \frac {\sqrt a +\sqrt b}{\sqrt 2}$ for any $a\ge 0$, $b\ge 0$,
we have
\begin{align*}
\sqrt{\frac 23 (\frac 12+\sqrt{\frac 2 {\tau}})} \ge \sqrt{\frac 16}
+\sqrt{ \frac 13 \sqrt{\frac 2{\tau}}} 
> \frac 13
+\sqrt{ \frac 13 \sqrt{\frac 2{\tau}}}.
\end{align*}
Thus we need to verify
\begin{align*}
 \frac 13 \sqrt{\frac 2{\tau}} \ge 
 (2 \sqrt{2E_0} +3E_0)^2
 \cdot \frac 1 {4\pi^2}
\cdot (4.386 + 2\pi( \ln(r+\frac 1 {\sqrt 2}) ) ).
 \end{align*}
Now note that $\ln (\frac3 {4\pi^{\frac 32}}) \approx -2.00478<-2$
and $\sqrt{2} \frac {4\pi^{\frac 32}}{3} 
\approx 10.49974<10.5$, and
\begin{align*}
\ln(r+\frac 1{\sqrt 2})
&=\ln( {\sqrt 2} +\frac 3 {4\pi^{\frac 32}}
( (2+\frac 1 {\tau}) \sqrt{E_0} +3 E_0 ) ) \notag \\
&<-2 + \ln (10.5 + 2\sqrt{E_0} +3E_0+ \frac 1 {\tau} \sqrt {E_0} ) 
\notag \\
&\le -2 + \ln \frac{10.5}6 + \ln (2\sqrt{E_0} +3E_0+ \frac 1 {\tau} \sqrt {E_0} ),
\qquad \text{(by \eqref{e502} and \eqref{e3.15exl})} \notag \\
&\le -1.44 + \frac 1 2 \ln E_0 + \ln( 2+\frac 1 {\tau} + 3\sqrt{E_0}) \notag \\
& \le -1.44 + \frac 12 \ln E_0 + \ln(2+3\sqrt{E_0})
+\ln \frac 1 {\tau}. \qquad \text{(since $\frac 1 {\tau}\ge 2$)}
\end{align*}
Here in the second inequality above, we have used the simple inequality
\begin{align} \label{e3.15exl}
10.5+b \le \frac {10.5} 6 b, \qquad \text{if $b >\frac 83 \pi^{\frac 32}$}.
\end{align}

It then suffices for us to prove
\begin{align*}
 \frac 13 \sqrt{\frac 2{\tau}} \ge 
 (2 \sqrt{2E_0} +3E_0)^2
 \cdot \frac 1 {4\pi^2}
\cdot (4.386 + 2\pi\biggl(  -1.44 + \frac 12 \ln E_0 + \ln(2+3\sqrt{E_0})
+\ln \frac 1 {\tau} \biggr) ).
 \end{align*}
Denote $F_0= (2 \sqrt{2E_0} +3E_0)^2$. Note that
\begin{align*}
&\frac {4.386-2\pi\cdot 1.44}{\pi} \approx -1.48389<-1.45, \\
&\frac 12 \ln E_0 + \ln(2+3\sqrt{E_0}) =\ln(2\sqrt {E_0}+3E_0)<\frac 1 2 \ln F_0.
\end{align*}
We then need to show
\begin{align*}
\frac 1 3 \sqrt{\frac 2 {\tau}} \ge F_0 \cdot \frac 1 {4\pi}
(-1.45+\ln F_0 -2 \ln \tau),
\end{align*}
or equivalently
\begin{align*}
\frac{4\pi}3 \sqrt 2
\ge F_0 (-1.45+\ln F_0) \sqrt{\tau}
-2 F_0 \sqrt{\tau} \ln \tau.
\end{align*}

Note that $\frac{4\pi}3 \sqrt 2 \approx 5.92384$. Now we discuss two cases.

Case A: $0<F_0\le 1$.  
 It
is not hard to check that
\begin{align*}
\sup_{0<x\le \frac 12} \sqrt x \ln (\frac 1 x) \le 0.8.
\end{align*}
Clearly then for $0<\tau\le \frac 12$, we have
\begin{align*}
2 \sqrt{\tau} \ln (\frac 1 {\tau}) \le   1.6<5.9
\end{align*}
which is clearly ok for us.

Case B: $F_0>1$. In this case set
\begin{align*}
\sqrt{\tau}= \frac {\delta}{ F_0(1+\ln F_0)}.
\end{align*}
Then 
\begin{align*}
-2 F_0 \sqrt{\tau} \ln \tau= \frac{4\delta}{1+\ln F_0}
(-\ln \delta + \ln F_0 + \ln (1+\ln F_0) ).
\end{align*}
It is not difficult to check that
\begin{align*}
\sup_{x>1}\frac{\ln x + \ln (1+\ln x)} {1+\ln x} \le 1.2.
\end{align*}
Thus for $0<\delta\le 1$,
\begin{align*}
-2 F_0 \sqrt{\tau} \ln \tau \le -4 \delta \ln \delta +4.8 \delta
\end{align*}
Then
\begin{align*}
 F_0 (-1.45+\ln F_0) \sqrt{\tau}
-2 F_0 \sqrt{\tau} \ln \tau \le
-4 \delta \ln \delta +5.8 \delta \le 5.8, \quad \text{for any $0<\delta \le 1$}.
\end{align*}

Concluding from all cases, we obtain that it suffices to take
\begin{align*}
 \tau_{\operatorname{max}} =
\begin{cases}
\frac 12,\quad \text{if $F_0\le 1$};\\
\min\{\frac 12,\, \frac 1 { (F_0 (1+\ln F_0))^2}\}, \quad\text{if $F_0>1$},
\end{cases}
\end{align*}
where $F_0= (2\sqrt{2E_0} +3E_0)^2$. Thus \eqref{e500a} holds and this completes
the base step.

Step 2. Induction step. The main induction hypothesis is that for $n\ge 1$, 
\begin{align*}
& E_{n} \le E_{n-1}, \\
&\frac 3 2\|u^{n}\|_{\infty}^2 \le 
\frac 12 + \sqrt {\frac 2 {\tau}}.
\end{align*}
Clearly by using similar estimates as in Step 1 for $u^1$, one can check that
$u^{n+1}$ satisfies the same inequality and $E_{n+1} \le E_n$. This then completes
the induction step.

\end{proof}
\section{Proof for 2D and $\nu=1$: second approach}
Recall 
\begin{align} \label{semi_e2}
\frac{u^{n+1}-u^n}{\tau} = -   \Delta^2 u^{n+1} +\Delta\Pi_N ( f(u^n) ), \quad n\ge 0.
\end{align}
We rewrite it as
\begin{align*}
u^{n+1}= \frac {1+2\tau \Delta}{1+\tau \Delta^2} u^n
+ \frac {\tau \Delta}{1+\tau \Delta^2} \Pi_N ( (u^n)^3- 3u^n ).
\end{align*}

\begin{lem} \label{m400}

\begin{align*}
&\Bigl(\sum_{0\ne k \in \mathbb Z^2}
(\frac {1} {1+ \tau (2\pi |k|)^4 } )^2 \frac 1 {(2\pi |k|)^2}\Bigr)^{\frac 12}
\le
\begin{cases}
 (- \frac 1 {8\pi} \ln \tau +0.52)^{\frac 12}, \quad \text{if $0<\tau\le (4\pi)^{-4}$};\\
0.8803, \quad \text{if $\tau> (4\pi)^{-4}$}.
\end{cases}
\\
&\Bigl(\sum_{0\ne k \in \mathbb Z^2}
(\frac {1-2\tau (2\pi |k|)^2} {1+ \tau (2\pi |k|)^4 } )^2 \frac 1 {(2\pi |k|)^2}
\Bigr)^{\frac 12}
\le\; h(\tau)=
\begin{cases}
 (- \frac 1 {8\pi} \ln \tau +0.52)^{\frac 12}+0.00872, \quad \text{if $0<\tau\le (4\pi)^{-4}$};\\
0.8988, \quad \text{if $\tau> (4\pi)^{-4}$}.
\end{cases}
\end{align*}
\end{lem}

\begin{proof}
The natural cut-off is $k_0=k_0(\tau)= \frac 1 {2\pi} \tau^{-\frac 14}$.

Case 1: $k_0\ge 2$. This is equivalent to $0<\tau \le (4\pi)^{-4}<
4.02 \times 10^{-5}$. 
Then by \eqref{eA001}, we have
\begin{align*}
\sum_{0<|k|\le k_0}  &(1+\tau (2\pi |k|)^4)^{-2} |k|^{-2}
 \le \sum_{0<|k|\le k_0} |k|^{-2} 
\le 4.4 +2\pi \ln (k_0 +\frac 1 {\sqrt 2}) \notag \\
&=4.4 + 2\pi \ln k_0 + 2\pi \ln (1+\frac 1 {k_0\sqrt 2}) \le 
-\frac {\pi} 2\ln \tau -2\pi \ln (2\pi)+4.4 +2\pi\ln (1+\frac 1{2 \sqrt 2}) \notag \\
& \le -\frac {\pi} 2\ln \tau-5.24.
\end{align*}
Also by Lemma \ref{m005}, we have
\begin{align*}
\sum_{|k|> k_0} (1+\tau (2\pi |k|)^4)^{-2} |k|^{-2}
& \le \tau^{-2} (2\pi)^{-8} \sum_{|k|>k_0} |k|^{-10}
\le \tau^{-2} (2\pi)^{-7} \cdot \frac 1 8 (k_0-\frac 1 {\sqrt 2})^{-8} 
 \notag \\
 & = 2\pi \cdot \frac 18\cdot  (1- \frac 1 {k_0 \sqrt{2} } )^{-8}
 \le 25.76.
\end{align*}
Thus
\begin{align*}
(\sum_{0\ne k \in \mathbb Z^2} (1+\tau (2\pi |k|)^4)^{-2} (2\pi |k|)^{-2})^{\frac 12}
&\le \; \frac 1 {2\pi} \cdot ( -\frac {\pi} 2\ln \tau +20.52 )^{\frac 12}
\le (- \frac 1 {8\pi} \ln \tau +0.52)^{\frac 12}.
\end{align*}

Now observe (for the first inequality we use $1+\tau(2\pi |k|)^4\ge 2 \sqrt{\tau} (2\pi |k|)^2$)
\begin{align*}
&\sum_{0<|k|\le k_0}\Bigl(
\frac {2\tau (2\pi |k|)^2} {1+ \tau (2\pi |k|)^4 } \Bigr)^2  |k|^{-2}
 \le \tau (-\frac {\pi} 2\ln \tau-5.24)<0.0015, \quad\text{for any $0<\tau<4.02\times 10^{-5}$};\\
&\sum_{|k|> k_0}\Bigl(
\frac {2\tau (2\pi |k|)^2} {1+ \tau (2\pi |k|)^4 } \Bigr)^2  |k|^{-2}
 \le 4 \cdot (2\pi)^{-4} \sum_{|k|>k_0} |k|^{-6} \le
  (2\pi)^{-3}  \cdot (2-\frac 1 {\sqrt 2})^{-4} <0.0015.
\end{align*}
Thus
\begin{align*}
\Bigl( 
&\sum_{0\ne k \in \mathbb Z^2}\Bigl(
\frac {2\tau (2\pi |k|)^2} {1+ \tau (2\pi |k|)^4 } \Bigr)^2  (2\pi |k|)^{-2}
 \Bigr)^{\frac 12} 
 \le \frac 1 {2\pi} (0.0015+0.0015)^{\frac 12} <0.00872.
 \end{align*}

Case 2: $0<k_0<2$. In this regime $\tau>(4\pi)^{-4}$. 

\begin{align*}
\sum_{0<|k|<2}
(1+\tau (2\pi |k|)^4)^{-2} |k|^{-2}
& = 4 (1+ \tau (2\pi)^4)^{-2} +4 \cdot \frac 12 \cdot (1+\tau (2\pi \sqrt 2)^4)^{-2} 
 \notag \\
 &\le 4 (1+\frac 1{16})^{-2} + 2 \cdot (1+\frac 14)^{-2} \le 4.83.
 \end{align*}
 
 \begin{align*}
\sum_{|k|\ge 2}
(1+\tau (2\pi |k|)^4)^{-2} |k|^{-2}
& \le \tau^{-2} (2\pi)^{-8} \sum_{|k|\ge 2} |k|^{-10}
\le (4\pi)^8 (2\pi)^{-7} \cdot \frac 1 8 (2-\frac 1 {\sqrt 2})^{-8} 
 \notag \\
 & = 2\pi \cdot \frac 18\cdot  (1- \frac 1{2 \sqrt{2} } )^{-8}
 \le 25.76.
\end{align*}
Thus
\begin{align*}
\Bigl(
\sum_{0\ne k \in \mathbb Z^2}
(1+\tau (2\pi |k|)^4)^{-2} (2\pi |k|)^{-2}\Bigr)^{\frac 12}
& \le \frac 1 {2\pi} \sqrt{4.83+25.76} < 0.8803.
\end{align*}

Now observe
\begin{align*}
\sum_{0<|k|< 2}\Bigl(
\frac {2\tau (2\pi |k|)^2} {1+ \tau (2\pi |k|)^4 } \Bigr)^2  |k|^{-2}
 &= 4 \tau^2 (2\pi)^4 \Bigl( 4 (1+ \tau (2\pi)^4)^{-2} +8  (1+\tau (2\pi \sqrt 2)^4)^{-2}\Bigr) \\
 &=\pi^{-4} x^2 ( (1+x)^{-2} +2 (1+4x)^{-2})<0.012, \quad \text{for any $x>\frac 1 {16}$},
 \end{align*}
 where we have denoted $x= (2\pi)^4 \tau >\frac 1{16}$. 
 
 On the other hand, we have
 \begin{align*}
&\sum_{|k|\ge  2}\Bigl(
\frac {2\tau (2\pi |k|)^2} {1+ \tau (2\pi |k|)^4 } \Bigr)^2  |k|^{-2}
 \le 4 \cdot (2\pi)^{-4} \sum_{|k|\ge 2} |k|^{-6} \le
  (2\pi)^{-3}  \cdot (2-\frac 1 {\sqrt 2})^{-4} <0.0015.
\end{align*}

Thus
\begin{align*}
\Bigl( 
&\sum_{0\ne k \in \mathbb Z^2}\Bigl(
\frac {2\tau (2\pi |k|)^2} {1+ \tau (2\pi |k|)^4 } \Bigr)^2  (2\pi |k|)^{-2}
 \Bigr)^{\frac 12} 
 \le \frac 1 {2\pi} (0.012+0.0015)^{\frac 12} <0.0185.
 \end{align*}

Finally to estimate $\Bigl(\sum_{0\ne k \in \mathbb Z^2}
(\frac {1-2\tau (2\pi |k|)^2} {1+ \tau (2\pi |k|)^4 } )^2 \frac 1 {(2\pi |k|)^2}
\Bigr)^{\frac 12}$, we just  use the triangle inequality $$\| A+B\|_{l_k^2}
\le \| A\|_{l_k^2} +\|B\|_{l_k^2}.$$

\end{proof}

\begin{thm}[Conditional energy stability for 2D $\nu=1$, second approach] \label{thmnu1b}
Let $d=2$, $\nu=1$ and $N\ge 2$.  Assume $u_0\in H^1(\mathbb T^2)$ and
has zero mean. Recall $u^0=\Pi_N u_0$ and assume $\|u^0\|_{\infty} \le 1.25$.
Take
\begin{align*}
 \tau_{\operatorname{max}} =
\min\Bigl\{ \frac 12, \;   \frac 1 {B_1 (\ln B_1)^2},\;  \frac 89 \frac 1 {B_2} \Bigr\},
\end{align*}
where 
\begin{align*}
&B_1= \frac {3.1}{99.8} (\sqrt {2E_0}  + 0.19676 E_0)^{4}, \\
& B_2=(0.8988 \sqrt {2E_0} + 0.18692 E_0)^4. 
\end{align*}

Then for any $0<\tau\le \tau_{\operatorname{max}}$, the scheme
\eqref{semi_e2} is energy stable, i.e.  
\begin{align*}
E(u^{n+1}) \le E(u^n), \qquad\forall\, n\ge 0.
\end{align*}

\end{thm}

\begin{proof}
The proof is similar to that in Theorem \ref{thmnu1a} and we only sketch the needed
modifications.  Note that since by assumption $0<\tau\le \frac 12$, we have
\begin{align*}
\frac 12 + \sqrt{\frac 2 {\tau}} \ge 2.5 \ge \frac 32 \|u^0\|_{\infty}^2.
\end{align*}
We shall focus on the induction step (since the estimate for $u^1$ is the same as the
estimate for $u^{n+1}$ below).
Our inductive hypothesis is  for $n \ge 1$, 
\begin{align*}
&E_n \le E_{n-1}; \\
&\| u^n \|_{\infty}
\le \sqrt{2E_0} h(\tau) + 0.18692 E_0, 
\end{align*}
where $h(\tau)$ is the same as in Lemma \ref{m400}. By Lemma \ref{m400}, we have
\begin{align*}
\| \frac {1+2\tau \Delta}{1+\tau \Delta^2} u^n \|_{\infty}
 \le h(\tau) \| u^n \|_{\dot H^1} \le \sqrt{2E_0} h(\tau).
 \end{align*}
 On the other hand by Lemma \ref{m011} and \ref{m100}, we have
 \begin{align*}
\|\frac {\tau \Delta}{1+\tau \Delta^2} \Pi_N  ( (u^n)^3- 3u^n )\|_{\infty}
\le \frac{\sqrt{6.05}} {(2\pi)^2} \cdot 3E_0 \le 0.18692 E_0.
\end{align*}
Clearly then 
\begin{align*}
\| u^{n+1}\|_{\infty}
\le \sqrt{2E_0} h(\tau) + 0.18692 E_0.
\end{align*}
It remains for us to verify $E_{n+1} \le E_n$. This amounts to checking 
\begin{align} \notag 
\frac 12 + \sqrt {\frac 2 {\tau}} \ge \frac 32 \max\{ \| u^n \|_{\infty}^2, \,
\| u^{n+1}\|_{\infty}^2 \}.
\end{align}
It suffices for us to prove
\begin{align*}
\sqrt{ (\frac 12 +\sqrt{ \frac 2 {\tau}} ) \cdot \frac 23 }
\ge \sqrt {2E_0} h(\tau) + 0.18692 E_0.
\end{align*}

Case 1: $\tau>(4\pi)^{-4}$. We only need to show
\begin{align*}
\sqrt{ (\frac 12 +\sqrt{ \frac 2 {\tau}} ) \cdot \frac 23 }
\ge \; 0.8988 \sqrt {2E_0} + 0.18692 E_0,
\end{align*}
or in a slightly simpler form:
\begin{align*}
\sqrt{\frac 2 {\tau}} \ge 
(0.8988 \sqrt {2E_0} + 0.18692 E_0)^2 \cdot \frac 32 - \frac 12.
\end{align*}
Thus it is sufficient to require
\begin{align*}
\tau \le \frac 89  \frac 1 {A_1},
\end{align*}
where
\begin{align*}
A_1=(0.8988 \sqrt {2E_0} + 0.18692 E_0)^4.
\end{align*}

Case 2: $0<\tau \le (4\pi)^{-4}$.  Then we need to show
\begin{align*}
\sqrt{ (\frac 12 +\sqrt{ \frac 2 {\tau}} ) \cdot \frac 23 }
\ge \sqrt {2E_0} \underbrace{\Bigl((- \frac 1 {8\pi} \ln \tau +0.52)^{\frac 12}+0.00872\Bigr)}_{=:f_0(\tau)} + 0.18692 E_0.
\end{align*}
It is not difficult to check that
\begin{align*}
f_0(\tau) \ge 0.95, \qquad \forall\;0<\tau\le (4\pi)^{-4}.
\end{align*}
Also $0.18692/0.95<0.19676$. Thus it suffices for us to show
\begin{align*}
\sqrt{ (\frac 12 +\sqrt{ \frac 2 {\tau}} ) \cdot \frac 23 } \cdot \frac 1 {f_0(\tau)}
\ge (\sqrt {2E_0}  + 0.19676 E_0).
\end{align*}
We only need to show
\begin{align*}
\frac 8 9 \cdot (\sqrt {2E_0}  + 0.19676 E_0)^{-4} \ge 
\tau f_0(\tau)^4.
\end{align*}
It is not difficult to check that
\begin{align*}
f_0(\tau) \le 1.54 (-\frac 1{8\pi} \ln \tau)^{\frac 12}, \qquad \forall\, 0<\tau\le (4\pi)^{-4}.
\end{align*}
We then only need to prove
\begin{align*}
99.8 (\sqrt {2E_0}  + 0.19676 E_0)^{-4} \ge 
\tau (\ln \tau)^2.
\end{align*}
Now discuss two cases.

Case a): $E_0 \ge 4.57916$. In this case we have
\begin{align*}
99.8 (\sqrt {2E_0}  + 0.19676 E_0)^{-4} \le 0.419538< 3.1 e^{-2}\approx 0.419539,
\end{align*}
By using Lemma \ref{lm4.3a} below, it suffices for us to require
\begin{align*}
0<\tau \le  \frac 1 {B_1 (\ln B_1)^2},
\end{align*}
where 
\begin{align*}
B_1= \frac {3.1}{99.8} (\sqrt {2E_0}  + 0.19676 E_0)^{4}.
\end{align*}

Case b): $E_0< 4.57916$. 
In this case we have
\begin{align*}
99.8 (\sqrt {2E_0}  + 0.19676 E_0)^{-4} \ge  0.418.
\end{align*}
It suffices to require $0<\tau\le 0.04$ since
\begin{align*}
\sup_{0<\tau\le 0.04} \tau (\ln \tau)^2 < 0.41445<0.418.
\end{align*}
Now recall that we are in the sub-case $0<\tau<(2\pi)^{-4} <0.04$, so this condition is
certainly satisfied.

\end{proof}

\begin{lem} \label{lm4.3a}
Consider $h(x)= x (\ln x)^2$ for $0<x \le 1 $.  If $0<x\le e^{-2}$,
then $h^{\prime}(x)\ge 0$.   If $A \ge e^2/3.1\approx 2.38357$ and
\begin{align*}
0<x \le \frac 1 {B (\ln B)^2},\quad B=3.1A,
\end{align*}
then 
\begin{align*}
h(x) \le \frac {1} A.
\end{align*} 
\end{lem}
\begin{proof}
The monotonicity of $h$ is easy to check.
It is also not difficult to check that if $B\ge e^2$, then
\begin{align*}
(\frac { \ln (B (\ln B)^2)} { \ln B} )^2= (1+\frac{2\ln\ln B}{\ln B})^2 \le 3.1.
\end{align*}
By monotonicity of $h$, we then have for any $0<x\le  \frac 1 {B(\ln B)^2}$,
\begin{align*}
h(x) \le h( \frac 1 {B(\ln B)^2})= \frac 1 B (1+\frac{2\ln\ln B}{\ln B})^2 \le
 \frac {3.1} B.
\end{align*}
\end{proof}

\section{Proof for general $\nu>0$: 2D case}
In this section we consider the general case $\nu>0$ in 2D. 
Recall 
\begin{align} \label{semi_e3}
\frac{u^{n+1}-u^n}{\tau} = - \nu  \Delta^2 u^{n+1} +\Delta\Pi_N ( f(u^n) ), \quad n\ge 0.
\end{align}
We rewrite it as
\begin{align*}
(1+ \nu\tau \Delta^2) u^{n+1} = u^n + \tau \Delta \Pi_N  ( f(u^n) ).
\end{align*}

\begin{lem} \label{lem2Dtau1}
Let $N\ge 2$, $d=2$ and $\nu>0$. Let $\tau>0$. Then for any $g\in L^4(\mathbb T^2)$
 with zero mean,
we have
\begin{align*}
&\| (1+\nu \tau \Delta^2)^{-1}  g \|_{\infty} \le C_1 (\nu \tau)^{-\frac 18}
\|g\|_4; 
\end{align*}
For any $g_1\in L^{\frac 43}(\mathbb T^2)$, we have
\begin{align*}
& \| \tau \Delta (1+\nu \tau \Delta^2)^{-1} \Pi_N g_1\|_{\infty}
\le C_2  \tau (\nu \tau)^{-\frac 78} \| g_1\|_{\frac 43}.
\end{align*}
In the above $C_1>0$, $C_2>0$ are absolute constants.

\end{lem}

\begin{rem*}
For the first estimate a similar estimate holds if the spectral projection $\Pi_N$ is  present. 
In our application later we do not need it since $u^n$ is already spectrally localized.
The operator $\Pi_N$ can  also be replaced by more general projection operators.
\end{rem*}
\begin{rem*}
Remarkably if we use the $\dot H^1$-norm which seems to be stronger, it will
incur a logarithmic loss for $\nu$.  The adoption of $L^4$ (for the homogeneous term)
and $L^{\frac 43}$ (for the inhomogeneous term) removes this divergence. 
A further refinement is possible by using $\dot H^1$ for the high frequency piece.  
This will lower some of the exponents on the constants such as $\alpha_1$, $\alpha_2$ 
in the proof of Theorem \ref{thm2dnu} below. However for simplicity of presentation we shall
not dwell on this issue here.
\end{rem*}

\begin{proof}
Denote $\beta=\nu \tau$.  The first inequality follows from Lemma
\ref{leKbeta} (see the bound for $\tilde K$ therein). For the second inequality denote 
\begin{align*}
K_{\beta}=\mathcal F^{-1}(
\frac {\beta^{\frac 12} (2\pi |k|)^2} {1+ \beta (2\pi |k|)^4} 1_{|k|_{\infty} \le N} ).
\end{align*}
We then have
\begin{align*}
\| K_{\beta} \|_4 \le \| \widehat{K_{\beta} } \|_{l_k^{\frac 43}} \lesssim
\; \beta^{-\frac 38}.
\end{align*}

\end{proof}

\begin{lem} \label{l_pot1}
Let $d\ge 1$.
If $E_p= \int_{\mathbb T^d} \frac 14 (v^2-1)^2 dx$, then
\begin{align*}
&\| v\|_{L^4(\mathbb T^d)} \le \sqrt{1+2 \sqrt{E_p}};\\
&\| v^3 -v \|_{L^{\frac 43}(\mathbb T^d)} \le  2\sqrt {E_p} \sqrt{1+2 \sqrt{E_p}}.
\end{align*}

\end{lem}
\begin{proof}
Obvious.
For the second inequality, note that 
$\| (v^2-1)v \|_{\frac 43}\le \| v^2-1\|_2 \|v\|_4$.
\end{proof}

\begin{thm}[Conditional energy stability for 2D $\nu>0$] \label{thm2dnu}
Let $d=2$, $\nu>0$, $N\ge 2$. Assume $u_0 \in H^1(\mathbb T^2)$ and
has zero mean. Assume $\| u^0\|_{\infty} =L_0<\infty$. Take
\begin{align*}
\tau_{\operatorname{max}}
= \min\Bigl\{ \frac {8 \nu} {9 L_0^4}, \, 
\tau_{\operatorname{max}}^{(1)} \Bigr\},
\end{align*}
where
\begin{align*}
 \tau_{\operatorname{max}}^{(1)} 
= 0.04\nu^3 \min\{\alpha_1^{-8},\, \alpha_2^{-\frac 83} \},
\end{align*}
and 
\begin{align*}
&\alpha_1= C_1 \sqrt{1+2\sqrt{E_0}};\\
& \alpha_{2}= 2C_2 \sqrt{E_0} \cdot  \sqrt{1+2 \sqrt{E_0}}.
\end{align*}
In the above $C_1$, $C_2$ are the same constants in Lemma \ref{lem2Dtau1}.
Then for any $0<\tau\le \tau_{\operatorname{max}}$, the scheme
\eqref{semi_e3} is energy stable, i.e.  
\begin{align*}
E(u^{n+1}) \le E(u^n), \qquad\forall\, n\ge 0.
\end{align*}

\end{thm}

\begin{proof}
We shall use induction.

Step 1. The base step $n=0$. 
Thanks to our choice of $\tau$, we clearly have for $n=0$,
\begin{align*}
\sqrt{\frac {2\nu}{\tau}} \ge \frac 32 \| u^0\|_{\infty}^2.
\end{align*}
To ensure $E_1\le E_0$, we need to check 
\begin{align*}
\sqrt{\frac {2\nu}{\tau}} \ge \frac 32 \| u^1\|_{\infty}^2.
\end{align*}

By Lemma \ref{lem2Dtau1} and Lemma \ref{l_pot1}, we have
\begin{align*}
\|u^{1}\|_{\infty}
\le \alpha_1 (\nu \tau)^{-\frac 18} + 
\alpha_2  \tau (\nu \tau)^{-\frac 78} ,
\end{align*}
where 
\begin{align*}
&\alpha_1= C_1 \sqrt{1+2\sqrt{E_0}};\\
& \alpha_{2}= 2C_2 \sqrt{E_0} \cdot  \sqrt{1+2 \sqrt{E_0}}.
\end{align*}

By Lemma \ref{z2}, to show $E_1\le E_0$, it suffices for us to check the inequality
\begin{align*}
 \Bigl(\frac 2 3\sqrt{\frac {2\nu} {\tau}}\Bigr)^{\frac 12}
 \ge \alpha_1 (\nu \tau)^{-\frac 18} + 
\alpha_2  \tau (\nu \tau)^{-\frac 78}.
\end{align*}
Now set $\tau = \nu^3 r$ and for $r$ we need to check
\begin{align*}
2^{\frac 34} 3^{-\frac 12} \ge \alpha_1 r^{\frac 18} +\alpha_2 r^{\frac 38}.
\end{align*}
Let $r=\alpha_1^4 \alpha_2^{-4} z$. Then we need
\begin{align*}
2^{\frac 34} 3^{-\frac 12} \ge \alpha_1^{\frac 32} \alpha_2^{-\frac 12}
(z^{\frac 18} +z^{\frac 38}).
\end{align*}
Now we choose $\eta\in (0,1)$, such that
\begin{align*}
&z^{\frac 18} \le \eta 2^{\frac 34} 3^{-\frac 12} \alpha_1^{-\frac 32} \alpha_2^{\frac 12};\\
&z^{\frac 38} \le (1-\eta) 2^{\frac 34} 3^{-\frac 12} \alpha_1^{-\frac 32} \alpha_2^{\frac 12}.
\end{align*}
A nearly optimal choice is $\eta=0.690119$ for which
\begin{align*}
&(\eta 2^{\frac 34} 3^{-\frac 12})^8 \approx 0.0406525>0.04;\\
&((1-\eta) 2^{\frac 34} 3^{-\frac 12})^{\frac 83}
\approx 0.0406523>0.04.
\end{align*}
Thus it suffices to require
\begin{align*}
z \le 0.04 \min\{ \alpha_1^{-12}\alpha_2^4,\, \alpha_1^{-4} \alpha_2^{\frac 43}
\}.
\end{align*}
Thus for $\tau$ we need
\begin{align*}
0<\tau\le \tau_{\operatorname{max}}^{(1)}
= 0.04\nu^3 \min\{\alpha_1^{-8},\, \alpha_2^{-\frac 83} \}.
\end{align*}
This in turn guarantees that $E_1 \le E_0$.

Step 2. Induction.  For $n\ge 1$, the induction hypothesis is that 
\begin{align*}
& E_{n} \le E_{n-1}, \\
&\|u^{n}\|_{\infty}
\le \alpha_1 (\nu \tau)^{-\frac 18} + 
\alpha_2  \tau (\nu \tau)^{-\frac 78}.
\end{align*}
Clearly by using similar estimates as in Step 1 for $u^1$, one can check that
\begin{align*}
&\|u^{n+1}\|_{\infty}
\le \alpha_1 (\nu \tau)^{-\frac 18} + 
\alpha_2  \tau (\nu \tau)^{-\frac 78},
\end{align*}
and 
\begin{align*}
\sqrt{\frac {2\nu}{\tau}} \ge \frac 32 \max\{ \|u^n\|_{\infty}^2, \,
\|u^{n+1}\|_{\infty}^2 \}.
\end{align*}
Then by Lemma \ref{z2} we obtain $E_{n+1} \le E_n$ which completes the induction
step. 

\end{proof}

\section{Proof for general $\nu>0$: 1D and 3D  case}
In this section we sketch the needed modifications for the 1D and 3D cases. The 1D
case will be similar to the 2D case. On the other hand the analysis for the 3D case
will be slightly different since $\dot H^1$ is no longer a critical case.

We first consider the 1D case.

 \begin{lem}[1D case] \label{lem700a}
Let $N\ge 2$, $d=1$ and $\nu>0$. Let $\tau>0$. Then for any $g\in L^4(\mathbb T)$ with mean zero,
we have
\begin{align*}
&\| (1+\nu \tau \Delta^2)^{-1}  g \|_{L^{\infty}(\mathbb T)}  \le B_1 (\nu \tau)^{-\frac 1{16}}
\|g\|_{L^4(\mathbb T)}.
\end{align*}
For any $g_1 \in L^{\frac 43}(\mathbb T)$, we have
\begin{align*}
& \| \tau \Delta (1+\nu \tau \Delta^2)^{-1} \Pi_N g_1\|_{L^{\infty}(\mathbb T)} 
\le B_2  \tau (\nu \tau)^{-\frac {11} {16} } \| g_1\|_{L^{\frac 43}(\mathbb T)}.
\end{align*}
In the above $B_1>0$, $B_2>0$ are absolute constants.

\end{lem}
\begin{proof}
Denote $\beta=\nu \tau$.  The first inequality follows from the bound of $\tilde K$ in 
Lemma \ref{leKbeta}. For the second inequality denote (since we are in 1D, $|k|=|k|_{\infty}$)
\begin{align*}
K_{\beta}=\mathcal F^{-1}(
\frac {\beta^{\frac 12} (2\pi |k|)^2} {1+ \beta (2\pi |k|)^4} 1_{|k| \le N} ).
\end{align*}
We then have
\begin{align*}
\| K_{\beta} \|_4 \le \| \widehat{K_{\beta} } \|_{l_k^{\frac 43}} \lesssim
\; \beta^{-\frac 3{16} }.
\end{align*}

\end{proof}
 
 \begin{thm}[Conditional energy stability for 1D $\nu>0$] \label{thm1dnu}
Let $d=1$, $\nu>0$, $N\ge 2$.
Assume $u_0 \in H^1(\mathbb T)$ and
has zero mean. Assume $\| u^0\|_{\infty} =L_0<\infty$. Take
\begin{align*}
\tau_{\operatorname{max}}
= \min\Bigl\{ \frac {8 \nu} {9 L_0^4}, \, 
\tau_{\operatorname{max}}^{(1)} \Bigr\},
\end{align*}
where
\begin{align*}
 \tau_{\operatorname{max}}^{(1)}
= 0.118\nu^{\frac 53}  \min \{ \beta_1^{-\frac {16}3},\, \beta_2^{-\frac{16}9} \},
\end{align*}
and
\begin{align*}
&\beta_1= B_1 \sqrt{1+2\sqrt{E_0}};\\
& \beta_{2}= 2B_2 \sqrt{E_0} \cdot  \sqrt{1+2 \sqrt{E_0}}.
\end{align*}
Here $B_1$, $B_2$ are the same constants in Lemma \ref{lem700a}.
Then for any $0<\tau\le \tau_{\operatorname{max}}$, the scheme
\eqref{semi_e3} is energy stable, i.e.  
\begin{align*}
E(u^{n+1}) \le E(u^n), \qquad\forall\, n\ge 0.
\end{align*}

\end{thm}

 \begin{proof}
 The induction procedure is similar to that in the proof of Theorem \ref{thm2dnu} and
 therefore we shall only sketch the needed modifications. The main inequality to
 verify is 
 \begin{align*}
 \Bigl(\frac 2 3\sqrt{\frac {2\nu} {\tau}}\Bigr)^{\frac 12}
 \ge \; \max\{ \|u^n\|_{\infty}, \, \|u^{n+1}\|_{\infty} \}.
 \end{align*}
The estimate of $\|u^n\|_{\infty}$ uses induction hypothesis. For $u^{n+1}$ we use
Lemma \ref{lem700a} and this gives
\begin{align*}
\|u^{n+1}\|_{\infty} & \le B_1 (\nu \tau)^{-\frac 1{16}}
\|u^n\|_{L^4(\mathbb T)} +B_2  \tau (\nu \tau)^{-\frac {11} {16} } \| f(u^n)\|_{L^{\frac 43}(\mathbb T)} \notag \\
& \le \beta_1 (\nu \tau)^{-\frac 1{16}}+\beta_2  \tau (\nu \tau)^{-\frac {11} {16} },
\end{align*}
 where in the second inequality we used  Lemma \ref{l_pot1}, and
 \begin{align*}
&\beta_1= B_1 \sqrt{1+2\sqrt{E_0}};\\
& \beta_{2}= 2B_2 \sqrt{E_0} \cdot  \sqrt{1+2 \sqrt{E_0}}.
\end{align*}

Set $\tau= \nu^{\frac 53} r$, and for $r$ we need to check the inequality
\begin{align*}
2^{\frac 34} 3^{-\frac 12} \ge \beta_1 r^{\frac 3{16}} + \beta_2 r^{\frac 9 {16}}.
\end{align*}
 We shall choose $r$ such that 
 \begin{align*}
 r \le z\cdot \min \{ \beta_1^{-\frac {16}3},\, \beta_2^{-\frac{16}9} \},
 \end{align*}
 and 
 \begin{align*}
 &z^{\frac 3 {16}} \le \, \eta 2^{\frac 34} 3^{-\frac 12}, \\
 & z^{\frac 9 {16}} \le \, (1-\eta) 2^{\frac 34} 3^{-\frac 12}.
 \end{align*}
 A nearly optimal choice is $\eta=0.690119$ for which
\begin{align*}
&(\eta 2^{\frac 34} 3^{-\frac 12})^{\frac {16}3} \approx 0.118229>0.118;\\
&((1-\eta) 2^{\frac 34} 3^{-\frac 12})^{\frac {16} 9}
\approx 0.118229>0.118.
\end{align*}
Thus it suffices to require
\begin{align*}
r \le 0.118 \min \{ \beta_1^{-\frac {16}3},\, \beta_2^{-\frac{16}9} \}.
\end{align*}
 
 \end{proof}

 The next lemma is for the 3D case. Note that the argument is slightly different
  from 2D and in some sense simpler.
 \begin{lem}[3D case] \label{l800b}
 Let $N\ge 2$, $d=3$, $\nu>0$ and $\tau>0$. 
Assume $g\in H^1(\mathbb T^3)$ has mean zero on $\mathbb T^3$. Then
\begin{align*}
&\| (1+\nu \tau \Delta^2)^{-1} g\|_{L^{\infty}(\mathbb T^3)}
 \le B_3 (\nu \tau)^{-\frac 18}
\|g\|_{\dot H^1(\mathbb T^3)}. 
\end{align*}
Let $N\ge 2$.  For any $g_1\in L^2(\mathbb T^3) $, we have
\begin{align*}
& \| \tau \Delta (1+\nu \tau \Delta^2)^{-1} \Pi_N g_1\|_{L^{\infty}(\mathbb T^3)}\le B_4  \tau (\nu \tau)^{-\frac 78} \| g_1\|_{L^2(\mathbb T^3)}.
\end{align*}
In the above $B_3>0$, $B_4>0$ are absolute constants. Also we have
\begin{align*}
& \| \tau \Delta (1+\nu \tau \Delta^2)^{-1} \Pi_N g_1\|_{L^{\infty}(\mathbb T^3)}\le B_4  \tau (\nu \tau)^{-\frac 78} \| g_1 -\overline {g_1}\|_{L^2(\mathbb T^3)},
\end{align*}
where $\overline{g_1}$ denotes the average of $g_1$ on $\mathbb T^3$.
\end{lem}
\begin{rem*}
The second estimate also holds when $\Pi_N$ is not present.
\end{rem*}
\begin{rem*}
Compared with 2D, here the argument is slightly simpler since $\dot H^1$ is no longer
a critical space (for $L^{\infty}$) we can make use of $L^2$ techniques.
\end{rem*}
\begin{rem*}
For the inhomogeneous estimate, one should note that 
\begin{align*}
\| \Delta (1+\Delta^2)^{-1} \delta_0 \|_{L^4 (\mathbb T^3)}=\infty
\end{align*}
and this is why we have to proceed differently from the 2D case.
\end{rem*}

\begin{rem*}
In the proof below we do not consider the refined bounds for $\beta=\nu \tau>1$ since one
is primarily interested in the case $0<\nu \ll 1$ and $0<\tau \lesssim 1$.
\end{rem*}
\begin{proof}
Denote $\beta=\nu \tau$. For the first inequality it suffices to check that
\begin{align*}
\bigl( \frac 1 {1+\beta (2\pi |k|)^4} \cdot \frac 1 {2\pi |k|} \bigr)_{l_k^2(0\ne k 
\in \mathbb Z^3)} \lesssim \; \beta^{-\frac 18}.
\end{align*}
If $\beta\ge 1$ the inequality is obvious since we have the stronger bound $\beta^{-1}$ in this case.
If $0<\beta<1$, one can then split into regimes $|k| \le \beta^{-\frac 14}$ and $|k|>\beta^{-\frac 14}$ and estimate separately the contributions. The bound $\beta^{-\frac 18}$ is then
immediate. Note that we can even calculate explicit constants here but we shall not dwell on this
issue here.

The proof of the second inequality is similar. We use
\begin{align*}
\Bigl( \frac {\beta^{\frac 12} (2\pi |k|)^2} { 1+ \beta (2\pi |k|)^4} \Bigr)_{l_k^2
(\mathbb Z^3)} \lesssim\; \beta^{-\frac 38}.
\end{align*}

\end{proof}

\begin{thm}[Conditional energy stability for 3D $\nu>0$] \label{thm3dnu}
Let $d=3$, $\nu>0$, $N\ge 2$.
 Assume $u_0 \in H^1(\mathbb T^3)$ and
has zero mean. Assume $\| u^0\|_{\infty} =L_0<\infty$. Take
\begin{align*}
\tau_{\operatorname{max}}
= \min\Bigl\{ \frac {8 \nu} {9 L_0^4}, \, 
\tau_{\operatorname{max}}^{(1)} \Bigr\}.
\end{align*}
 Here
 \begin{align*}
\tau_{\operatorname{max}}^{(1)}
=0.0007 \nu^7 \min \{ \beta_3^{-8},\, \beta_4^{-\frac{8}3},\, \beta_5^{-8} \},
\end{align*}
where 
\begin{align*}
&\beta_3= B_3^{\prime} \sqrt{E_0};\quad \beta_4 = B_4^{\prime} E_0^{\frac 32}; \\
& \beta_{5}= B_5^{\prime} (1+E_0)^{\frac 14},
\end{align*}
and $B_3^{\prime}>0$, $B_4^{\prime}>0$, $B_5^{\prime}>0$ are some
absolute constants.
Then for any $0<\tau\le \tau_{\operatorname{max}}$, the scheme
\eqref{semi_e3} is energy stable, i.e.  
\begin{align*}
E(u^{n+1}) \le E(u^n), \qquad\forall\, n\ge 0.
\end{align*}

\end{thm}

 \begin{proof}
 The induction procedure is similar to that in the proof of Theorem \ref{thm2dnu} and
 therefore we shall only sketch the needed modifications. By Lemma \ref{z2}, the main inequality to
 verify is 
 \begin{align*}
 \Bigl(\frac 13+\frac 2 3\sqrt{\frac {2\nu} {\tau}}\Bigr)^{\frac 12}
 \ge \; \max\{ \|u^n\|_{\infty}, \, \|u^{n+1}\|_{\infty} \}.
 \end{align*}
The estimate of $\|u^n\|_{\infty}$ uses induction hypothesis. For $u^{n+1}$ we use
Lemma \ref{l800b} and this gives
\begin{align*}
\|u^{n+1}\|_{\infty} & \le B_3 (\nu \tau)^{-\frac 1{8}}
\| \nabla u^n\|_{L^2(\mathbb T^3)} +B_2  \tau (\nu \tau)^{-\frac {7} {8} } \| f(u^n)
\|_{L^{2}(\mathbb T^3)} \notag \\
& \le \beta_3  \nu^{-\frac 12} (\nu \tau)^{-\frac 1{8}}+\beta_4  
\tau (\nu \tau)^{-\frac {7} {8} } \nu^{-\frac 32} +\beta_5 
\tau (\nu \tau)^{-\frac {7} {8} },
\end{align*}
 where in the second inequality we used the fact that
 \begin{align*}
 &\| (u^n)^3 \|_{L^2(\mathbb T^3)} \lesssim \| \nabla u^n \|_{L^2(\mathbb T^3)}^3
 \lesssim \nu^{-\frac 32} E_0^{\frac 32}; \\
 &\|u^n \|_{L^2(\mathbb T^3)} \lesssim 
 (1+E_0)^{\frac 14},
 \end{align*}
and for some absolute constant $\tilde B_4>0$, $B_5>0$, 
 \begin{align*}
&\beta_3= B_3 \sqrt{2 E_0};\quad \beta_4 = \tilde B_4 E_0^{\frac 32}; \\
& \beta_{5}= B_5 (1+E_0)^{\frac 14}.
\end{align*}

Set $\tau= \nu^{7} r$, and for $r$ we need to check the inequality
\begin{align*}
(\frac 13 \nu^3 r^{\frac 12}+ \frac 23 \sqrt{ 2} )^{\frac 12} \ge \beta_3 r^{\frac 1{8}} + \beta_4 r^{\frac 3 {8}}
+\beta_5 r^{\frac 38} \nu^{\frac 32}.
\end{align*}

Now by using the inequality
\begin{align*}
\sqrt{a+b} \ge \frac {\sqrt a +\sqrt b} {\sqrt 2}, \qquad \forall\, a,\, b\ge 0,
\end{align*}
it suffices for us to choose $r$ such that
\begin{align*}
&(  \frac {\sqrt 2} 3 )^{\frac 12} \ge \beta_3 r^{\frac 1{8}} + \beta_4 r^{\frac 3 {8}}; 
\\
&(\frac 16 \nu^3 r^{\frac 12} )^{\frac 12} \ge \beta_5 r^{\frac 38} \nu^{\frac 32}.
\end{align*}
The second inequality requires that
\begin{align*}
r \le 6^{-4} \beta_5^{-8}.
\end{align*}
Note that $6^{-4} \approx 0.000771605>0.0007$.
For the first inequality we shall choose $r$ such that 
 \begin{align*}
 r \le z\cdot \min \{ \beta_3^{-8},\, \beta_4^{-\frac{8}3} \},
 \end{align*}
 and 
 \begin{align*}
 &z^{\frac 1 {8}} \le \, \eta 2^{\frac 14} 3^{-\frac 12}, \\
 & z^{\frac 3 {8}} \le \, (1-\eta) 2^{\frac 14} 3^{-\frac 12}.
 \end{align*}
 A nearly optimal choice is $\eta=0.778006$ for which
\begin{align*}
&(\eta 2^{\frac 14} 3^{-\frac 12})^{8} \approx 0.0066288>0.006;\\
&((1-\eta) 2^{\frac 14} 3^{-\frac 12})^{\frac 83}
\approx 0.0066288>0.006.
\end{align*}
Thus it suffices to require
\begin{align*}
r \le 0.0007 \min \{ \beta_3^{-8},\, \beta_4^{-\frac{8}3},\, \beta_5^{-8} \}.
\end{align*}
 
 \end{proof}

\section{Concluding remarks}
Implicit-Explicit (IMEX) methods can simulate efficiently many phase field models such
such as the Cahn-Hilliard equation or thin film type equations. 
Compared with pure 
explicit methods, IMEX is more stable with larger allowable time steps
 whilst being efficient and accurate. In contrast with implicit methods and
partially implicit methods, 
IMEX  does not require solving a nonlinear system at each time step and is much more
efficient. In numerical experiments  IMEX is often observed to be energy stable provided
the time step is not taken too large. Due to the difficulties caused by the lack of maximum
principle and stiffness caused by the effect of small viscosity,
 the rigorous stability analysis of IMEX methods was a long standing open 
problem.
In this work we analyzed a model IMEX scheme introduced by Chen and Shen
\cite{CS98} for the Cahn-Hilliard equation and gave 
a first rigorous proof of conditional energy stability with mild time step constraints.
Our analysis does not rely on adding additional stabilization terms, truncating the nonlinearity or introducing auxiliary variables. 
To deal with the aforementioned difficulties caused by the lack of maximum
principle and stiffness  of small viscosity, we introduce
a  Trade-Energy-For-$L^{\infty}$ (TEFL) method  which is a refinement of our earlier work
\cite{LQT17, LQ17, LQ173d}.  In the course of the proof we  computed explicitly (and nearly optimal
in terms of energy scaling) time step constraints 
in several model cases which seem to be the first done
in the literature. All these developments are pivotal for future refined analysis on these algorithms.
Our theoretical analysis  shows that IMEX is a robust algorithm  for large scale
and long-time simulations, 
due to its simplicity and guaranteed conditional
energy stability with affordable time step constraints. 
It is expected that this new streamlined TEFL proof can be further refined and 
adapted to higher order cases
and generalized to many other phase field models and settings. 

\frenchspacing
\bibliographystyle{plain}

\end{document}